          \documentclass{cmslatex}
\usepackage[paperwidth=7in, paperheight=10in, margin=.875in]{geometry}
\usepackage{ifthen, amsmath,amssymb,graphicx, mathrsfs, amsfonts}
\usepackage{esint,bm,color}

\renewcommand{\theequation}{\arabic{section}.\arabic{equation}}

            \newtheorem{thm}{Theorem}[section]
          \newtheorem{prop}[thm]{Proposition}
          \newtheorem{lem}[thm]{Lemma}

          \newtheorem{rem}[thm]{Remark}

       	 \newcommand{\R}{\mathbb{R}}
	 
	 \newcommand{\E}{\mathrm{E}}

\begin{document}
          \title{Metastable dynamics for hyperbolic variations of the Allen--Cahn equation\thanks{This work was partially supported by the Italian Project FIRB 2012 ``Dispersive dynamics: Fourier Analysis and Variational Methods''.}}


          \author{Raffaele Folino\thanks{Dipartimento di Ingegneria e Scienze dell'Informazione e Matematica, Universit\`a degli Studi dell'Aquila,
          							via Vetoio, 67010 Coppito (L'Aquila) AQ, Italy,
          							(raffaele.folino@univaq.it).}
          \and{Corrado Lattanzio}\thanks{ Dipartimento di Ingegneria e Scienze dell'Informazione e Matematica, Universit\`a degli Studi dell'Aquila, 		
          via Vetoio, 67010 Coppito (L'Aquila) AQ, Italy,
          							(corrado@univaq.it),}
								{http://people.disim.univaq.it/$\thicksim$corrado/.}
          \and{Corrado Mascia}\thanks{Dipartimento di Matematica,
		Sapienza Universit\`a di Roma, P.le Aldo Moro, 2, 00185 Roma, Italy, (mascia@mat.uniroma1.it), }
		{http://www1.mat.uniroma1.it/$\sim$mascia/.}}





         \pagestyle{myheadings} \markboth{Metastable dynamics for hyperbolic variations of the Allen--Cahn equation}{Raffaele Folino, Corrado Lattanzio and Corrado Mascia} 
         
         \maketitle

          \begin{abstract}
	Metastable dynamics of a hyperbolic variation of the Allen--Cahn equation with homogeneous Neumann
	boundary conditions are considered.
	Using the ``dynamical approach'' proposed by Carr--Pego \cite{Carr-Pego} and Fusco--Hale \cite{Fusco-Hale} to study 
	slow-evolution of solutions in the classic parabolic case, we prove existence and persistence of metastable patterns for
	an exponentially long time. 
	In particular, we show the existence of an ``approximately invariant'' $N$-dimensional manifold $\mathcal{M}_{{}_0}$
	for the hyperbolic Allen--Cahn equation: if the initial datum is in a tubular neighborhood of $\mathcal{M}_{{}_0}$,
	the solution remains in such neighborhood for an exponentially long time. 
	Moreover, the solution has $N$ transition layers and the transition points move with exponentially small velocity. 
	In addition, we determine the explicit form of a system of ordinary differential equations describing the motion of
	the transition layers and we analyze the differences with the corresponding motion valid for the parabolic case.
          \end{abstract}
          


\section{Introduction}\label{introduction}

\subsection{Metastability in reaction-diffusion equations}
Reaction-diffusion equations are widely used to describe a variety of phenomena such as pattern formation
and front propagation in biological, chemical and physical systems.
When the model under study is characterized by the presence of competing equilibrium states, a crucial question is to
describe the interaction and the dynamics of space occupation by the equilibria.
A basic prototype is the \emph{Allen--Cahn equation}, which has the form
\begin{equation}\label{AllenCahnMulti}
	u_t+\mathcal{L}(u)=0\qquad\textrm{where}\quad
	\mathcal{L}(u):=-\varepsilon^2 \Delta u+f(u),
\end{equation}
and it has been originally proposed in \cite{Allen-Cahn} to describe the motion of antiphase boundaries in iron alloys. 
Such equation has an associated energy functional
\begin{equation*}
	\int \bigg\{\tfrac{1}{2} \varepsilon^2|\nabla u|^2+F(u)\bigg\}dx,
\end{equation*}
where the {\it potential} $F$ is a primitive of $f$ and integration in space is performed in the domain under consideration.
For the Allen--Cahn model, the function $F$ is assumed to be a double-well potential with wells of equal depth.
As a consequence, the reaction function $f$ has a cubic-type behavior with two stable and one unstable
equilibria, usually normalized as $\pm 1$ and $0$, respectively.

In the absence of diffusion, viz. $\varepsilon=0$, the space variable $x$ becomes an external parameter and   
solutions generically converge pointwise to functions with values in $\{-1,+1\}$ with sharp transition layers
generated at points where the initial datum changes sign.
For small $\varepsilon>0$, if the initial datum is a small perturbation of a function with values in $\{-1,+1\}$
with well-separated transition regions, diffusion determines in a short time-scale a smoothed version of the original
configuration and, on a longer time-scale, layers interact giving rise to front motion.
When the space variable is one-dimensional, starting from \cite{Bron-Kohn,Carr-Pego,Fusco-Hale},
it has been shown that, as long as layers are well-separated the interaction force is very weak and the
consequent motion is very slow.
The meaning of weak/slow can be quantified more precisely, as discussed in what follows.
Postponing such details, in the regime $\varepsilon\to 0^+$, the original configuration is preserved for a long time
and thus such behavior has been classified as {\it metastability}.

Many papers have been devoted to slow motion analysis for the Allen--Cahn equation providing precise description
of the relation between the size of the diffusivity $\varepsilon$ and the time-scale of the dynamics. 
A complete list of references would be prohibitive. 
Here, we only quote the analysis on generation, persistence and
annihilation of metastable patterns performed in \cite{Chen}.
A large class of different evolution PDEs, concerning many different areas, exhibits the phenomenon of metastability.
Without claiming to be complete, we list some of the principal models that have been analyzed:
scalar conservation laws \cite{FLMS17, KreiKrei86, LafoOMal95, Mascia-Strani, ReynWard95}, 
the Cahn--Hilliard equation \cite{AlikBateFusc91, Bates-Xun1, Bates-Xun2, Bron-Hilh, Pego89},
Gierer--Meinhardt and Gray--Scott systems \cite{Sun-War-Rus}, 
Keller--Segel chemotaxis models \cite{Dol-Sch, Pop-Hillen}, 
general gradient flows \cite{Otto-Rez}, high-order systems \cite{Kal-VdV-Wan},
gradient systems with equal depth multiple-well potentials \cite{Be-Or-Sm,Be-Sm},
Cahn--Morral systems \cite{Grant}, the Jin--Xin system \cite{Strani}.

The aforementioned bibliography is confined to one-dimensional models; 
however, there is a vast literature of works about motion of interfaces in several space dimensions, 
where the effect of the curvature of the interfaces turns out to be relevant for the dynamics. 
In particular, for the Allen--Cahn equation, we recall the works \cite{Bron-Kohn2,Chen2,demot-sch}, where it has been
shown that steep interfaces are generated in a short time with subsequent motion governed by mean curvature flow. 

The present paper is devoted to the analysis of metastability in a hyperbolic framework.
Precisely, given $\varepsilon>0$, $\tau_0\in (0,+\infty]$, $f\,:\,\mathbb{R}\rightarrow\mathbb{R}$
and $g\,:\,\mathbb{R}\times(0,\tau_0) \rightarrow\mathbb{R}$, 
we consider here the \emph{hyperbolic Allen--Cahn equation}
\begin{equation}\label{hyp-al-caMulti}
	\tau u_{tt}+g(u,\tau)u_t=\varepsilon^2 \Delta u-f(u).
\end{equation}
The function $f$ is required to be the derivative of a double well potential $F$ with non-degenerate minima 
of same depth, and the function $g$ is assumed to be strictly positive, uniformly with respect to $u$;
namely, we assume
\begin{align}
	&F(\pm1)=F'(\pm1)=0, \quad F''(\pm1)>0, \quad F(u)>0 \; \mbox{ for } u\neq\pm1, \label{hypf}\\
	&g(u,\tau)\geq c_g>0\qquad\forall\, u, \label{hypg}
\end{align}
where the constant $c_g$ may depend on $\tau$.
The uniform positivity of $g$ in \eqref{hypg} is crucial, because it guarantees the dissipative nature of the model.
If, in addition, $g(u,\tau)\to 1$ as $\tau\rightarrow0$, we formally recover \eqref{AllenCahnMulti}
from \eqref{hyp-al-caMulti} in the (singular) limit. 

In the case $g\equiv 1$, equation \eqref{hyp-al-caMulti} can be obtained by adding a nonlinear zero order
perturbation to the damped wave operator $\tau\partial_t^2 +\partial_t - \varepsilon^2\Delta$.
For such a choice, many studies have been devoted to the stability of fronts ---mainly in one space dimension--- for bistable
or monostable reaction term $f$ (see \cite{Gallay-Joly} and references therein).
Interface formation has been analyzed in \cite{Hil-Nara} in the singular limit $\varepsilon\to 0$ for space dimension
equal to $2$ or $3$, showing that motion is governed by mean curvature flow, as is the case for the corresponding
parabolic model.

The choice of the hyperbolic variation \eqref{hyp-al-caMulti}  is motivated by the observation that there are
different ways for modeling transport mechanisms. 
The one at the base of \eqref{AllenCahnMulti} is the classical Fourier law, originally proposed
for heat conduction and then extended to many other different fields, which prescribes
the instantaneous proportionality between the flux $v$ of a quantity with ``density'' $u$ and its gradient, $v = - \varepsilon^2\nabla u$.
Such choice has the advantage of providing a simple equation enjoying a number of useful properties (smoothing
effects, self-similarity, \dots), but, at the same time, it has a number of drawbacks, the best known being the presence
of infinite speed of propagation.
Still in the framework of heat conduction modeling, following some ideas developed by Maxwell in the context of kinetic theories,
Cattaneo proposed in  \cite{Cat} a different law for the heat flux $v$, based on the assumption that the equilibrium between flux
and gradient of the unknown is asymptotical with a time-scale measured by the relaxation parameter $\tau>0$, that is
\begin{equation}\label{MClaw}
	\tau v_t + v = - \varepsilon^2 \nabla u
	\qquad \textrm{(Maxwell--Cattaneo law)}
\end{equation}
(an extensive discussion is reported in \cite{JP89a,JP89b}).
In the one-dimensional case, the diffusion equation given by the law \eqref{MClaw} has also a probabilistic interpretation,
appearing in the description of \emph{correlated random walks} (see \cite{Gol, Kac, Tay}) to be compared with the
standard \emph{random walk}, which gives raise to the standard parabolic diffusion equation.

Criticisms to the application of the use of Fourier-type law has been given also in modeling reaction-diffusion
phenomena (see \cite{Hade99, Holmes}).
In the presence of a reaction term described by the function $f$, application of the Maxwell--Cattaneo law gives
\begin{equation*}
	\tau u_{tt}+\bigl\{u+\tau f(u)\bigr\}_t=\varepsilon^2 \Delta u-f(u),
\end{equation*}
to be considered as a modification of the standard Allen--Cahn equation when $f$ satisfies \eqref{hypf}
(see also \cite{DunbOthm86}, for different origins of the same equation).
This equation fits into \eqref{hyp-al-caMulti} with the choice $g(u,\tau):=1+\tau f'(u)$.
We refer to this specific model as the \emph{Allen--Cahn equation with relaxation}, reminiscent
of the relaxation-type law \eqref{MClaw}.
Existence and nonlinear stability of traveling wave solutions for this equation has been analyzed
in detail in \cite{LMPS} for general bistable reaction terms in one space dimension.

\subsection{Presentation of the main result}
This study is devoted to the one-dimensional case, so that the hyperbolic Allen--Cahn equation reads as
\begin{equation}\label{hyp-al-ca}
	\tau u_{tt}+g(u,\tau)u_t+\mathcal{L}(u)=0,
\end{equation}
corresponding to the parabolic Allen--Cahn equation
\begin{equation}\label{AllenCahn}
	u_t+\mathcal{L}(u)=0,
\end{equation}
where $\mathcal{L}(u):=-\varepsilon^2 u_{xx}+f(u)$.
Specifically, we are interested in the limiting behavior of the solutions as $\varepsilon\rightarrow0$ with the aim of extending
the metastable dynamics for \eqref{AllenCahn} to the hyperbolic case \eqref{hyp-al-ca}
and focusing the attention on eventual differences.

In \cite{Folino}, adapting the \emph{energy approach} proposed by Bronsard and Kohn \cite{Bron-Kohn} for the parabolic
equation \eqref{AllenCahn}, the first author has shown that, if the initial profile $u_0$ has a transition layer structure and
the initial velocity $v_0$ is small, then the solution maintains the transition layer structure on a time scale of order $\varepsilon^{-k}$
with $k$ arbitrary.
The energy approach has also been applied to Cahn--Hilliard equation in \cite{Bron-Hilh}.
Grant \cite{Grant} improved this method to prove exponentially slow motion for Cahn--Morral systems. 

A different procedure, proposed by Carr--Pego in \cite{Carr-Pego} and Fusco--Hale in \cite{Fusco-Hale}, permits to prove existence and persistence of metastable
states for the Allen--Cahn equation \eqref{AllenCahn} for a time proportional to $e^{C/\varepsilon}$.
This strategy provides also an explicit differential equation for the dynamics of the transition layer positions (far from collapses).
The method is based on the construction of an $N$-dimensional base manifold $\mathcal{M}$ consisting of functions which
approximate metastable states with $N$ transition layers.
The manifold is not invariant, but if the initial datum is in a small neighborhood of $\mathcal{M}$, then the solution remains
near the manifold for a time proportional to $e^{C/\varepsilon}$.
Based on these ideas, slow motion results have been proved for the Cahn--Hilliard equation
by Alikakos et al. \cite{AlikBateFusc91} and by Bates and Xun \cite{Bates-Xun1,Bates-Xun2}.
In particular, the last ones use the same manifold constructed in \cite{Carr-Pego}.

Here, we adapt the method of \cite{Carr-Pego} to the hyperbolic Allen--Cahn equation \eqref{hyp-al-ca} embedding
the base manifold $\mathcal{M}$ in an extended phase space determined by the presence of the additional unknown $v=u_t$
as suggested by the first-order form of equation \eqref{hyp-al-ca} given by
\begin{equation}\label{system-u-v}
	\left\{\begin{aligned}
		& u_t  = v\\
		& \tau v_t  = -\mathcal{L}(u)-g(u,\tau)v.
		\end{aligned}\right.
\end{equation}
System \eqref{system-u-v}, considered here for $t>0$ and $x\in(0,1)$, is complemented with homogeneous Neumann
boundary conditions
\begin{equation}\label{Neumann}
	u_x(0,t)=u_x(1,t)=0, \qquad \quad t>0,
\end{equation}
and initial conditions 
\begin{equation}\label{initial}
	u(x,0)=u_0(x), \quad v(x,0)=v_0(x), \qquad \quad x\in (0,1).
\end{equation}
The initial-boundary value problem \eqref{system-u-v}-\eqref{Neumann}-\eqref{initial} 
is globally well-posed for positive times in $H^1\times L^2$.
In particular, if 
\begin{equation*}
	(u_0,v_0)\in\mathcal{D}=\left\{(u,v)\in H^2\times H^1: u_x(0)=u_x(1)=0\right\},
\end{equation*}
the solution $(u,v)$ is classical and belongs to $C\left([0,\infty),\mathcal{D}\right)\cap C^1\left([0,\infty),H^1\times L^2\right)$
(among others, see \cite[Appendix A]{Folino}).
Then, our aim is to describe the dynamics of such globally defined solution, at least for a class of
``well-prepared'' initial data.

Under assumptions \eqref{hypf}--\eqref{hypg}, the hyperbolic equation \eqref{hyp-al-ca} supports traveling wave
solutions connecting the equilibria $-1$ and $1$, i.e. solutions of the form $u(x,t)=\Phi(x-ct)$ such that $\Phi(\pm\infty)=\pm 1$,
if and only if $c=0$. 
Indeed, substituting the traveling wave ansatz in the equation, we obtain
\begin{equation*}
	(\varepsilon^2-c^2\tau)\Phi''+c\,g(\Phi,\tau)\Phi'-f(\Phi)=0,
\end{equation*}
and thus, multiplying by $\Phi'$ and integrating over $\R$, we get
\begin{equation*}
	c\int_{\R} g(\Phi,\tau)(\Phi')^2\,d\xi=F(+1)-F(-1),
\end{equation*}
from which we deduce, under \eqref{hypf} and \eqref{hypg}, that the velocity $c$ is zero.
With such choice, it is well-known that, up to translation, there is a unique solution to the problem
\begin{equation}\label{Fi(x)}
	\varepsilon^2\Phi''-f(\Phi)=0, \qquad
	\Phi(x)\rightarrow\pm1 \quad \mbox{ as } \quad x\rightarrow\pm\infty.
\end{equation}
Normalizing $\Phi$ by adding the condition $\Phi(0)=0$, transitions layer from $-1$ to $+1$ (or viceversa)
are described by $\Phi(\pm(x-\bar x))$ for both equations \eqref{AllenCahn} and \eqref{hyp-al-ca}.

Steady states $\Phi$ are at the base of the construction of the base manifold $\mathcal{M}$
which we sketch here (for precise definitions, see Section \ref{preliminaries}). 
Fix $N\in\mathbb{N}$ and $\varepsilon>0$.
Given a configuration $\bm{h}=(h_1,\dots,h_N)$ of $N$ layer positions (with $h_j<h_{j+1}$), we construct a function $u^{\bm{h}}$
which approximates a metastable state with transition points at $h_1,\dots,h_N$, by piecing together approximated versions of $\Phi$,
that is $u^{\bm h}(x)\approx \Phi(x-h_j)$ or $\Phi(h_j-x)$ for $x\approx h_j$ (see Figure \ref{u^h-fig}).
\begin{figure}[htbp]
\includegraphics[width=13cm]{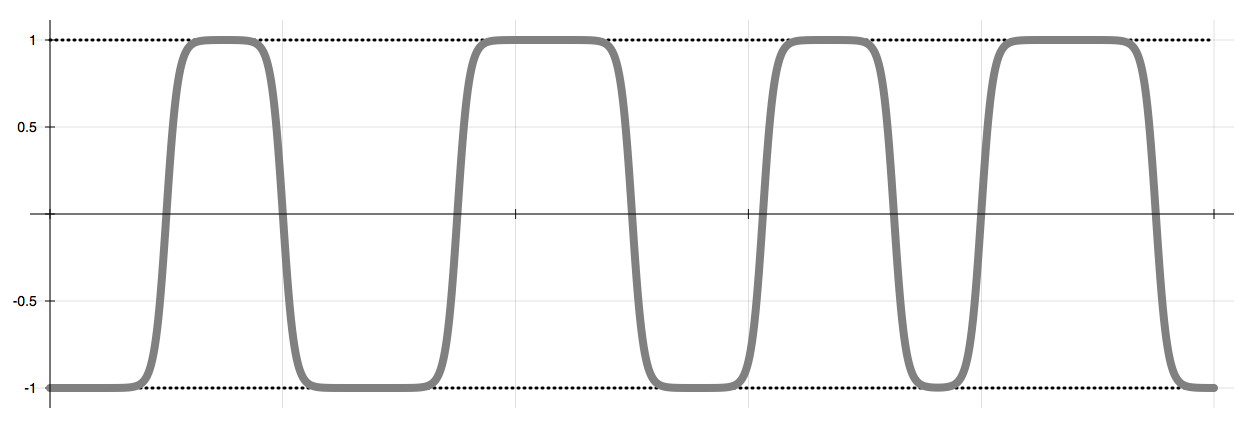}
\caption{Example of a function $u^{\bm h}(x)$ with $N=8$.
}
\label{u^h-fig}
\end{figure}

Then, we consider the slow evolution of solutions when the transition points are well separated one from the other and bounded
away from the boundary points $0$ and $1$.
For fixed (small) $\rho>0$, the admissible layer positions lie in the set
\begin{equation*}
	\Omega_\rho:=\bigl\{{\bm h}\in\mathbb{R}^N\, :\,0<h_1<\cdots<h_N<1,\;
		h_{j+1}-h_{j}>\varepsilon/\rho\mbox{ for } j=0,\dots,N\bigr\},
\end{equation*}
where $h_0:=-h_1$, $h_{N+1}:=2-h_N$ and the {\it base manifold} is $\mathcal{M}:=\{u^{\bm h}\,:\,{\bm h}\in\Omega_\rho\}$.

In what follows, we fix a minimal distance $\delta>0$ with $\delta<1/N$ and we consider the parameters $\varepsilon$
and $\rho$ such that
\begin{equation}\label{triangle}
	0<\varepsilon<\varepsilon_0\qquad\textrm{and}\qquad \delta<\frac{\varepsilon}{\rho}<\frac{1}{N},
\end{equation}
for some $\varepsilon_0>0$ to be chosen appropriately small.
In such a way, the parameters $\rho$ and $\varepsilon$ have the same order of magnitude.
All of the subsequent estimates depend on $N$ and $\delta$.

Denoted by $\langle\cdot,\cdot\rangle$ the inner product in $L^2(0,1)$,
to restrict the attention to a neighborhood of $\mathcal{M}$, we introduce the decomposition $u=u^{\bm h}+w$,
where $w$ are such that the following {\it orthogonality condition} holds
\begin{equation}\label{ortogonale}
	\langle w, k^{\bm h}_j\rangle=0, \qquad \quad \mbox{for } \quad j=1,\dots,N,
\end{equation} 
for some appropriate  approximate tangent vectors $k^{\bm h}_j$.
Then, setting
\begin{equation*}
	H^2_N:=\bigl\{w\in H^2(0,1)\, : \, w_x(0)=w_x(1)=0,\, \langle w,k^{\bm h}_j\rangle=0 \quad \mbox{ for } \; j=1,\dots,N\bigr\},
\end{equation*}
we consider triples $(\bm{h}, w, v)$ in the set $\Omega_\rho\times H^2_N\times L^2(0,1)$ and
the corresponding {\it extended base manifold}
\begin{equation*}
	\mathcal{M}_{{}_{0}}:=\mathcal{M}\times\{0\}=\left\{(u^{\bm h},0): u^{\bm h}\in{\mathcal{M}}\right\}.
\end{equation*}
Next, we choose a tubular neighborhood of $\mathcal{M}_{{}_{0}}$: given $\Gamma,\rho>0$, we set
\begin{equation*}
	\mathcal{Z}_{{}_{\Gamma,\rho}}:=\bigl\{(u,v)\,:\,u=u^{\bm h}+w,\,
	 ({\bm h},w,v)\in\overline{\Omega}_\rho\times H^2_N\times L^2(0,1),\,
	 \mathcal{E}^{\bm{h}}[w,v]\leq\Gamma \Psi({\bm h})\bigr\},
\end{equation*}
with the {\it energy functional} $\mathcal{E}^{\bm{h}}$ and the {\it barrier function} $\Psi$ defined by
\begin{align}
	\mathcal{E}^{\bm{h}}[w,v]&:=\tfrac12\int_0^1\bigl\{\varepsilon^2 w_x^2+f'(u^{\bm h})w^2\bigr\}dx
		+\tfrac12\tau\|v\|^2+\varepsilon\tau\langle w,v\rangle,	
		\label{energy} \\
	\Psi({\bm h})&:=\sum_{j=1}^N{\langle\mathcal{L}\bigl(u^{\bm h}\bigr),k^{\bm h}_j\rangle}^2,
		\label{barrier}
\end{align}
where $\|\cdot\|$ is the $L^2-$norm.
Our main result states that the channel $\mathcal{Z}_{{}_{\Gamma,\rho}}$ is invariant for an exponentially long time
if the parameters $\Gamma$ and $\rho$ are appropriately chosen.
In other words, the manifold $\mathcal{M}_{{}_{0}}$ is approximately invariant for the hyperbolic system \eqref{system-u-v}.
\vspace{0.2cm}
\begin{thm}\label{main}
Let $f\in C^2$ and $g(\cdot,\tau)\in C^1$ with $\tau\in (0,\tau_0)$ be such that $f=F'$ and \eqref{hypf}-\eqref{hypg} hold.
Given $N\in\mathbb{N}$ and $\delta\in(0,1/N)$, there exist $\Gamma_2>\Gamma_1>0$
and $\varepsilon_0>0$ (possibly depending on $\tau$) such that, if $\varepsilon,\rho$  satisfy \eqref{triangle},
$\Gamma\in[\Gamma_1,\Gamma_2]$ and the initial datum satisfies
\begin{equation*}
	(u_0,v_0)\in\,\stackrel{\circ}{\mathcal{Z}}_{{}_{\Gamma,\rho}}=\bigl\{(u,v)\in\mathcal{Z}_{{}_{\Gamma,\rho}}\, : \, {\bm h}\in\Omega_\rho
	\;\;\textrm{and}\;\; \mathcal{E}^{\bm{h}}[w,v]<\Gamma\Psi({\bm h})\bigr\},
\end{equation*}
then the solution $(u,v)$ to the initial-boundary value problem \eqref{system-u-v}-\eqref{Neumann}-\eqref{initial}
remains in $\mathcal{Z}_{{}_{\Gamma,\rho}}$ for a time $T_\varepsilon>0$, and there exists $C>0$ (possibly depending on $\tau$)  such that
for any $t\in[0,T_\varepsilon]$
\begin{align}
	\varepsilon^{1/2}\|w\|_{{}_{L^\infty}}+\|w\|+\tau^{1/2}\|v\|&\leq C\exp(-A\ell^{\bm h}/\varepsilon), \label{umenouh}\\
	|{\bm h}'|_{{}_{\infty}} &\leq C(\varepsilon/\tau)^{1/2}\exp(-A\ell^{\bm h}/\varepsilon), \label{|h'|<exp-intro}
\end{align}
where $A:=\sqrt{\min\{f'(-1),f'(1)\}}$, $\ell^{\bm h}:=\min\{h_j-h_{j-1}\}$ and $|\cdot|_{{}_{\infty}}$
denotes the maximum norm in $\mathbb{R}^N$.
Moreover, 
\begin{equation*}
		T_\varepsilon\geq C(\tau/\varepsilon)^{1/2}(\ell^{\bm h(0)}-\varepsilon/\rho)\exp(A\delta /\varepsilon).
\end{equation*}
\end{thm}
\begin{rem}\label{rem:added}
It is worth to observe that in the above theorem, and in general in the whole paper, $\tau$ should be viewed as a \emph{fixed}   parameter in $(0,\tau_0)$, and, as clearly stated, the constants may depend on it.
However, we prefer to make the ratio $\varepsilon/\tau$ appear in the estimates above because the constants  may be chosen uniform with respect to $\tau$ in many cases, as for the relaxation limit $\tau\rightarrow0$ from 
  the hyperbolic equation \eqref{hyp-al-ca} to the parabolic   Allen--Cahn equation \eqref{AllenCahn}, namely for $g(u,\tau)\rightarrow 1$ as $\tau\rightarrow0$; the main examples  in this framework we have already introduced above are $g\equiv 1$ and 
$g(u,\tau)=1+\tau f'(u)$.
  More precisely, 
if     $g(u,\tau)\to 1$ as $\tau\rightarrow0$ in any reasonable way and $u$ is  bounded, 
 then  \eqref{hypg} implies  $0<c_g\leq g(u,\tau)\leq C_g$  with $c_g$ and $C_g$ independent from $\tau$ in a (right) neighborhood of   zero.
With this extra (uniform in $\tau$) control at our disposal, one can follow the  proofs needed to obtain our main result and see that the only dependence in $\tau$ in the bounds for ${\bm h}'$ and $T_\varepsilon$ is through  the  aforementioned ratio $\varepsilon/\tau$, 
which can be used to study  the interplay between the two small parameters $\varepsilon$ and $\tau$ while performing this relaxation limit.
\end{rem}

The strategy to prove Theorem \ref{main} is the following.
Firstly, plugging the decomposition $u=u^{\bm h}+w$ into system \eqref{system-u-v} and using conditions \eqref{ortogonale},
we obtain an  ODE-PDE coupled system describing the dynamics for $({\bm h},w,v)$, see system \eqref{system-w-v-h}.
Then, we show that, if the solution $(u,v)$ belongs to $\mathcal{Z}_{{}_{\Gamma,\rho}}$, the estimates
\eqref{umenouh} and \eqref{|h'|<exp-intro} hold.
Next, we estimate the time $T_\varepsilon$ taken for the solution $(u,v)$ to leave the channel $\mathcal{Z}_{{}_{\Gamma,\rho}}$.
The boundary of $\mathcal{Z}_{{}_{\Gamma,\rho}}$ is the union of two parts: the ``ends'' where ${\bm h}\in\partial\Omega_\rho$,
meaning $h_j-h_{j-1}=\varepsilon/\rho$ for some $j$ and ``sides'' where $\mathcal{E}^{\bm{h}}[w,v]=\Gamma\Psi({\bm h})$.
Using an energy estimate, we infer that the solution can leave $\mathcal{Z}_{{}_{\Gamma,\rho}}$ only through the ends.
Since, for \eqref{|h'|<exp-intro}, the transition points move with exponentially small velocity, the solution $(u,v)$ stays
in the channel for an exponentially long time.

As long as the solution $(u,v)$ remains in the channel $\mathcal{Z}_{{}_{\Gamma,\rho}}$, $u$ is a function with $N$ transition layers.
The estimate \eqref{|h'|<exp-intro} ensures the slow motion of solutions and gives a lower bound on the lifetime of the metastable states. 
In order to give further information on the motion of the transition layers and an upper bound on such lifetime,
we study in detail an approximation of the equation for ${\bm h}$, determined formally by the requirement that
$u(x,t)=u^{\bm h(t)}(x)$ is an exact solution. 
Such a requirement is expected to be appropriate in the limit $\varepsilon\to 0$.
In this way, we obtain a system of ordinary differential
equations for ${\bm h}$ which does not depend on $w$ and $v$ and has the form
\begin{equation}\label{h-eq-intro}
	\tau {\bm h}''+\gamma_\tau {\bm h}'=\mathcal{P}^\ast({\bm h}),
\end{equation}
where $\gamma_\tau:=\overline{g}(\cdot,\tau)$ and
the (weighted) average $\overline{\mathtt{g}}$ of the continuous function $\mathtt{g}$ is given by
\begin{equation*}
	\overline{\mathtt{g}}:=\frac{1}{\|\sqrt{F}\|_{{}_{L^1}}}\int_{-1}^{1} \sqrt{F(s)}\,\mathtt{g}(s)\,ds,
\end{equation*}
and $\mathcal{P}^*$ is a function, depending on $F$. 
Equation \eqref{h-eq-intro} has to be compared with the corresponding one for the parabolic case \eqref{AllenCahn},
which is ${\bm h}'=\mathcal{P}^*({\bm h})$.
For the nonlinear damped wave equation $g\equiv 1$, we  have $\gamma_\tau =1$, while for the Allen--Cahn equation
with relaxation we obtain $\gamma_\tau=1+\tau \overline{f'}$.
Since $\overline{f'}$ is negative, the (physical relevant) relaxation case exhibits smaller friction effects with respect
to the damped one (details in Section \ref{layer}).

System \eqref{h-eq-intro} has a unique equilibrium point $(\bm{h}^e,0)$ where $\bm{h}^e$ is the unique zero
of $\mathcal{P}^*$, that corresponds to the unique stationary solution $u^e$ of \eqref{hyp-al-ca} with $N$
transition layers, normalized by the condition $u(0)<0$, without loss of generality.
In the parabolic case, $\bm{h}^e$ is an unstable equilibrium point with $N$ positive eigenvalues;
whereas, for the hyperbolic model, $(\bm{h}^e,0)$ is an unstable equilibrium point for \eqref{h-eq-intro} with $N$
positive eigenvalues and $N$ negative eigenvalues.

The rest of the paper is organized as follows.
In Section \ref{preliminaries} we give all the definitions, preliminaries and the construction of the manifold $\mathcal{M}$.
Furthermore, we recall all the results of Carr and Pego \cite{Carr-Pego} needed to prove Theorem \ref{main}.
Section \ref{motion} is devoted to the derivation of the equation of motion for the triple $(\bm{h},w,v)$
and to the proof of Theorem \ref{main}.
In Section \ref{layer}, we deduce the approximating equation for ${\bm h}$, we prove that there is a unique equilibrium
point $(\bm{h}^e,0)$ and we study its stability.
Finally, using singular perturbation theory, we show that, for $\tau$ small, if $g$ is uniformly bounded and
$g(u,\tau)\to 1$ a.e.\ as $\tau\rightarrow0$, the behavior of the solution to \eqref{h-eq-intro} is the same
of the parabolic case (see Theorem \ref{thm:tau0}). 

\section{Preliminaries}\label{preliminaries}
Following \cite{Carr-Pego}, we construct the base manifold and collect estimates
that we will use in the proof of our results.
For fixed $\rho>0$, we recall the definition
\begin{equation*}
	\Omega_\rho:=\bigl\{{\bm h}\in\mathbb{R}^N\, :\,0<h_1<\cdots<h_N<1,\quad
		 h_j-h_{j-1}>\varepsilon/\rho\mbox{ for } j=1,\dots,N+1\bigr\},
\end{equation*}
where $h_0:=-h_1$ and $h_{N+1}:=2-h_N$.
By construction, if $\rho_1<\rho_2$, then  $\Omega_{\rho_1}\subset \Omega_{\rho_2}$.

The idea is to associate to any $\bm h\in\Omega_\rho$ a function $u^{\bm h}=u^{\bm h}(x)$ which approximates a metastable
state with $N$ transition points at $h_1,\dots,h_N$ by matching appropriate steady states of equation \eqref{AllenCahn}.
The collection of $u^{\bm h}$ determines a $N$-dimensional manifold.
In order to describe the dynamics in a neighborhood of such manifold, the framework has to be complemented with a
projection which permits to separate the solution into a component on the manifold and a corresponding remainder.
For the Allen--Cahn equation \eqref{AllenCahn}, two different constructions have been proposed in \cite{Carr-Pego}
and \cite{Fusco-Hale}.

In \cite{Fusco-Hale}, Fusco and Hale use functions $\Phi(\pm(x-\bar x))$ with $\Phi$ the solution of \eqref{Fi(x)}
previously defined, and set
\begin{equation*}
	U^{\bm h}(x):=\Phi\left((x-h_j)(-1)^{j+1}\right), \quad x\in[h_{j-1/2},h_{j+1/2}], \quad j=1,\dots,N,
\end{equation*}
where
\begin{equation*}
	h_{j+1/2}:=\tfrac12(h_j+h_{j+1})\qquad j=0,\dots,N,
\end{equation*}
(note that $h_{1/2}=0$, $h_{N+1/2}=1$).
Hence, they obtain a manifold $\mathcal{M}^{\textrm{\tiny FH}}$ composed by continuous functions
$U^{\bm h}$ with a piecewise continuous first order derivative that jumps at $h_{j+1/2}$, $j=1,\dots,N-1$. 
In particular, the elements of the manifold belong to $H^1$ and not to $H^2$ (if $N>1)$.
In addition, they construct a tubular neighborhood of $\mathcal{M}^{\textrm{\tiny FH}}$ with coordinates $(\bm h,V)$
by setting 
\begin{equation*}
	u=U^{\bm h}+V\qquad\textrm{with}\quad \langle V,U^{\bm h}_j\rangle=0 \quad j=1,\dots, N,
\end{equation*}
where $U^{\bm h}_j$ are the derivatives of $U^{\bm h}$ with respect to $h_j$.
By construction, $U^{\bm h}_j$ have disjoint supports and $U^{\bm h}_j(x)=-U^{\bm h}_x(x)$ for all $x\in(h_{j-1/2},h_{j+1/2})$. 
In \cite{Fusco-Hale}, it is also conjectured that equation \eqref{AllenCahn} has an invariant manifold
$\mathcal{M}^{\textrm{\tiny FH}}_\ast$ near $\mathcal{M}^{\textrm{\tiny FH}}$ and that this manifold
$\mathcal{M}^{\textrm{\tiny FH}}_\ast$ is a graph over $\mathcal{M}^{\textrm{\tiny FH}}$.
Fusco and Hale did not prove the existence of the invariant manifold, but assuming existence, they calculated a
first approximation for $\mathcal{M}^{\textrm{\tiny FH}}_\ast$ and for the differential equation for $\bm{h}$ describing
the reduced flow.
They also conjectured that metastable states with $N$ transitions are associated with the unstable manifold of 
stationary solutions of \eqref{AllenCahn} having $N$ layers.

Both conjectures have been proved in \cite{Carr-Pego2} using a different base manifold, 
previously constructed in \cite{Carr-Pego}.
The approach used by Carr and Pego is based on a different choice and matching of steady states, which provides
functions $u^{\bm h}$, composing the base manifold $\mathcal{M}^{\textrm{\tiny CP}}$, which are smooth in both 
$x$ and $\bm{h}$.
The crucial difference with respect to the Fusco--Hale approach, resides in the fact that,
for $\mathcal{L}(u):=-\varepsilon^2 u_{xx}+f(u)$, 
\begin{equation*}
	\begin{aligned}
	&\mathcal{L}(U^{\bm h})=0	&\quad\textrm{and}\quad &\mathcal{L}(u^{\bm h})\neq 0
		&\qquad &\textrm{for}\quad x\approx h_j,\\
	&\mathcal{L}(U^{\bm h})\neq 0	&\quad\textrm{and}\quad &\mathcal{L}(u^{\bm h})=0	
		&\qquad &\textrm{for}\quad x\approx h_{j+1/2},
	\end{aligned}
\end{equation*}
with major consequences on the location of $\mathcal{L}(u^{\bm h})$ with respect to the tangent space to
$\mathcal{M}^{\textrm{\tiny CP}}$ at $u^{\bm h}$, as will be clear in the following presentation.

In this paper, we follow the framework established by Carr and Pego adapting it to the case of the hyperbolic
Allen--Cahn equation \eqref{hyp-al-ca}.
Since the equation we consider corresponds to the system \eqref{system-u-v}, the dynamics is determined by an additional unknown,
the time derivative $v=u_t$, and thus the base manifold $\mathcal{M}^{\textrm{\tiny CP}}$ has to be embedded in a extended 
vector space.
Here, taking advantage of the fact that we are looking for a manifold that is only approximately invariant, we perform this extension 
in a trivial way, considering the {\it extended base manifold}
$\mathcal{M}^{\textrm{\tiny CP}}_{{}_{0}}:=\mathcal{M}^{\textrm{\tiny CP}}\times\{0\}$.

From now on, we drop the letters $CP$ in the symbol used for the manifolds.

\subsection{Carr--Pego base manifold}
Given $L>0$, let $\varphi(\cdot,L,+1)$ be the solution to 
\begin{equation}\label{eq:z}
	-\varphi_{xx}+f(\varphi)=0, \qquad \quad
	\varphi\bigl(-\tfrac12L\bigr)=\varphi\bigl(\tfrac12L\bigr)=0,
\end{equation}
with $\varphi>0$ in $(-\tfrac12L,\tfrac12L)$, and let $\varphi(\cdot,L,-1)$ be the solution to \eqref{eq:z} with $\varphi<0$ in $(-\tfrac12L,\tfrac12L)$. 
Observe that if $\varphi$ satisfies \eqref{eq:z}, then 
\begin{equation}\label{eq:varfi}
	\varphi_x^2=2\{F(\varphi)-F(\varphi(0))\}.
\end{equation}
Using this formula, we can prove existence and uniqueness of the solutions $\varphi(\cdot,L,\pm1)$.

\begin{lem}
Let $f=F',$ with $F$ a smooth function satisfying \eqref{hypf}.
There exists $L_0>0$ such that, if $L>L_0$, then the functions $\varphi(\cdot,L,\pm1)$ are well-defined
and, denoting by 
\begin{equation*}
	M_\pm(L):=\max_x|\varphi(x,L,\pm1)|=|\varphi(0,L,\pm1)|,
\end{equation*} 
we have that $M_\pm$ is an increasing function of $L$ and $M_\pm(+\infty)=1$.
\end{lem}
This lemma is consequence of the fact that $\pm1$ are absolute minima of $F$ and so, 
there exist periodic solutions of \eqref{eq:varfi} oscillating around $0$. 
Indeed, the existence of such solutions is guaranteed if there exist $M_\pm\in(0,1)$ such that
$F(M_+)=F(-M_-)$, $F'(M_\pm)\neq0$ and $F(s)>F(M_+)$ for all $s\in(-M_-,M_+)$.
This condition is certainly satisfied if $M_\pm$ are close to $+1$.
Let us consider the positive case $\varphi(\cdot,L,+1)$ and $M_+(L)=\varphi(0,L,+1)$.
By integrating \eqref{eq:varfi} in $(-\tfrac12L,0)$ and using the boundary conditions in \eqref{eq:z}, 
we obtain
\begin{equation}\label{eq:L-per}
	L=\sqrt2\int_0^{M_+}\frac{ds}{\sqrt{F(s)-F(M_+)}}.
\end{equation}
The integral in \eqref{eq:L-per} tends to infinity as $M_+\to1^-$ 
and it is an increasing function of $M_+$ for $M_+$ close to $+1$. 
Hence, for $L$ sufficiently large, there exists a unique $M_+$ such that \eqref{eq:L-per} is satisfied
and so the function $\varphi(\cdot,L,+1)$ is well-defined. 
The negative case $\varphi(\cdot,L,-1)$ and $M_-(L)=-\varphi(0,L,-1)$ is similar.

Now, given $\ell>0$, let us define $\phi(x,\ell,\pm1):=\varphi\bigl(\frac x\varepsilon,\frac \ell\varepsilon,\pm1\bigr)$.
By definition, it follows that $\phi(\cdot,\ell,+1)$ is the solution to
\begin{equation}\label{fi(x,l)}
	\mathcal{L}(\phi):=-\varepsilon^2\phi_{xx}+f(\phi)=0, \qquad \quad
	\phi\bigl(-\tfrac12\ell\bigr)=\phi\bigl(\tfrac12\ell\bigr)=0,
\end{equation}
with $\phi>0$ in $(-\tfrac12\ell,\tfrac12\ell)$,
and $\phi(\cdot,\ell,-1)$ is the solution to \eqref{fi(x,l)} with $\phi<0$ in $(-\tfrac12\ell,\tfrac12\ell)$.
Moreover, the functions $\phi(\cdot,\ell,\pm1)$ are well-defined if $\ell>\varepsilon L_0$, they depend on $\varepsilon$ 
and $\ell$ only through the ratio $\varepsilon/\ell$. 
Finally,  
\begin{equation*}
	\max_x|\phi(\cdot,\ell,\pm1)|=M_\pm(\ell/\varepsilon)
	\qquad \quad \textrm{and} \qquad \quad
	\max_x|\phi_x(\cdot,\ell,\pm1)|\leq C\varepsilon^{-1},
\end{equation*}
where $C>0$ is a constant depending only on $F$. 
In particular, $M_\pm$ tends to $+1$ as $\varepsilon/\ell\to 0$ (more details in Proposition \ref{prop:alfa,beta}).

For $\bm h\in\Omega_\rho$ with $\rho<1/L_0$, we define the function $u^{\bm h}$ with $N$ transition points at $h_1,\dots,h_N$
by matching together steady states to \eqref{hyp-al-ca} with layer distance equal to $\ell$, using smooth cut-off functions.
Given $\chi:\mathbb{R}\rightarrow[0,1]$ a $C^\infty$ function with $\chi(x)=0$ for $x\leq-1$ and $\chi(x)=1$ for $x\geq1$,
set 
\begin{equation*}
	\chi^j(x):=\chi\left(\frac{x-h_j}\varepsilon\right) \qquad\textrm{and}\qquad
	\phi^j(x):=\phi\left(x-h_{j-1/2},h_j-h_{j-1},(-1)^j\right).
\end{equation*}
Then  the function $u^{\bm h}$ is given by the convex combination
\begin{equation}\label{u^h(x)}
	u^{\bm h}:=\left(1-\chi^j\right)\phi^j+\chi^j\phi^{j+1} \qquad \textrm{in}\quad I_j:=[h_{j-1/2},h_{j+1/2}],
\end{equation}
and the base manifold for the equation \eqref{AllenCahn} is defined as
\begin{equation*}
	\mathcal{M}:=\{u^{\bm h} :\bm h\in\Omega_\rho\}.
\end{equation*} 
If $\rho>0$ is sufficiently small and $\bm h\in\Omega_\rho$, then $u^{\bm h}(x)\approx\Phi\left((x-h_j)(-1)^{j-1}\right)$
for $x$ near $h_j$ and $u^{\bm h}(x)\approx\pm1$ away from $h_j$ for $j=1,\dots,N$. 
Therefore, states $u^{\bm h}$ on the base manifold are well approximated near transition layers by $U^{\bm h}$.

By definition, $u^{\bm h}$  is a smooth function of $x$ and $\bm h$ and enjoys the properties
\begin{equation*}
	\begin{aligned}
	u^{\bm h}(0)&=\phi(0,2h_1,-1)<0,
			&\qquad 	u^{\bm h}(h_{j+1/2})&=\phi\left(0,h_{j+1}-h_j,(-1)^{j+1}\right)\\
	u^{\bm h}(h_j)&=0,
			&\qquad \mathcal{L}(u^{\bm h}(x))&=0\quad \textrm{for }|x-h_j|\geq\varepsilon,
	\end{aligned}
\end{equation*}
for any $j=1,\dots,N$.
In what follows, we use the notation
\begin{equation*}
	u^{\bm h}_j:=\partial_{h_j} u^{\bm h}, \qquad \quad
	\nabla_{\bm h} u^{\bm h}:=\bigl(u^{\bm h}_1,\dots,u^{\bm h}_N\bigr),
\end{equation*}
and we denote the tangent space to $\mathcal{M}$ at $u^{\bm h}$ by $T\mathcal{M}(u^{\bm h})=\mbox{span}\{u^{\bm h}_j : j=1,\dots,N\}$. 
At this point, the natural idea would be to construct a tubular neighborhood of $\mathcal{M}$,
with coordinates $(\bm h,w)$ where $w$ is orthogonal to $T\mathcal{M}(u^{\bm h})$.
Since $\mathcal{M}$ is not invariant, there is higher flexibility in the construction of its neighborhood
and tubular co-ordinates near $\mathcal{M}$ can be defined using approximate tangent vectors to $\mathcal{M}$. 
For $j=1,\dots,N$, introduce the cutoff function $\gamma^j$ given by
\begin{equation*}
	\gamma^j(x):=\chi\left(\frac{x-h_{j-1/2}-\varepsilon}\varepsilon\right)\left[1-\chi\left(\frac{x-h_{j+1/2}+\varepsilon}\varepsilon\right)\right].
\end{equation*}   
Then, the {\it approximate tangent vectors} $k^{\bm h}_j$ are defined by
\begin{equation*}
	k^{\bm h}_j(x):=-\gamma^j(x)u^{\bm h}_x(x).
\end{equation*}
By construction, $k^{\bm h}_j$ are smooth functions of $x$ and $\bm h$ and are such that
\begin{equation*}
	\begin{aligned}
	k^{\bm h}_j(x)&=0				&\quad \textrm{for}\quad &x\notin[h_{j-1/2},h_{j+1/2}],\\
	k^{\bm h}_j(x)&=-u^{\bm h}_x(x)	&\quad \textrm{for}\quad &x\in[h_{j-1/2}+2\varepsilon,h_{j+1/2}-2\varepsilon]. 
	\end{aligned}
\end{equation*}
As above, we use the notation  
\begin{equation*}
	k^{\bm h}_{ji}:=\partial_{h_i} k^{\bm h}_j, \qquad \quad
	\nabla_{{}_{\bm{h}}} k^{\bm h}_j:=\bigl(k^{\bm h}_{j1},\dots,k^{\bm h}_{jN}\bigr).
\end{equation*}
 The definition of the approximate tangent vectors is motivated by the relations
\begin{equation*}
	\begin{aligned}
		u_j^{\bm h} = \partial_{h_j} u^{\bm h} \approx \partial_{h_j}\Phi\left((x-h_j)(-1)^{j-1}\right )
			= (-1)^{j} \Phi'\left((x-h_j)(-1)^{j-1}\right )\approx -u^{\bm h}_x
	\end{aligned}
\end{equation*}	
for $x\in[h_{j-1/2},h_{j+1/2}]$.	
In addition, the multiplication by the cutoff term $\gamma^j$ is reminiscent of the fact that the tangent space
of $\mathcal{M}^{\textrm{\tiny FH}}$ is spanned by $U_j^{\bm h}$ that have disjoint supports.

The following estimates will be useful in the sequel.
\vspace{0.2cm}
\begin{prop}[Carr--Pego \cite{Carr-Pego}] \label{estimates-u^h_j}
Let $f=F'$ with $F$ satisfying \eqref{hypf}.
Given $N\in\mathbb{N}$ and $\delta\in(0,1/N)$, there exist $\varepsilon_0, C, A_0>0$, 
and  a function $\omega=\omega(s)$ with $\omega\to 0$ as $s\rightarrow 0^+$
such that if $\varepsilon$ and $\rho$ are chosen so that \eqref{triangle} holds and $\bm{h}\in\Omega_\rho$, then
\begin{equation*}
	\|u^{\bm h}_j\|_{{}_{L^\infty}}+\varepsilon^{1/2}\|k^{\bm h}_{jj}\|+\|k^{\bm h}_{jj}\|_{{}_{L^1}}\leq C\varepsilon^{-1},
\end{equation*}
\begin{equation*}	
	\begin{aligned}
	A_0-\omega(\rho)&\leq \varepsilon^{1/2}\|u_j^{\bm h}\|\leq A_0+\omega(\rho),\\
	A_0-\omega(\rho)&\leq \varepsilon^{1/2}\|k_j^{\bm h}\|\leq A_0+\omega(\rho), \\
	\bigl\{A_0-\omega(\rho)\bigr\}^2&\leq \varepsilon\langle u_j^{\bm h},k^{\bm h}_j\rangle \leq\bigl\{A_0+\omega(\rho)\bigr\}^2, \\
	\end{aligned}
\end{equation*}
for $\bm h\in\Omega_\rho$ and $j=1,\dots,N$. Moreover, if $j\neq i$, we have
\begin{equation*}
	 |\langle u^{\bm h}_j,k^{\bm h}_i\rangle|+\varepsilon^{1/2}\|k^{\bm h}_{ij}\|
	 	+\|k^{\bm h}_{ij}\|_{{}_{L^1}}\leq \omega(\rho)\varepsilon^{-1}.
\end{equation*}
\end{prop}

\noindent
Heuristically, the exponent of $\varepsilon$ can be obtained by replacing
$u^{\bm h}_j$ and $k^{\bm h}_j$ with $-\Phi'(x-h_j)$. 

A function $u$ near $\mathcal{M}$ may be written in terms of coordinates $(\bm h,w)$ as $u=u^{\bm h}+w$,
with $w$ satisfying the orthogonality condition \eqref{ortogonale}. 
To state this result, let us set
\begin{equation*}
	\begin{aligned}
	\mathcal{B}_{\rho,\sigma}&:=\Bigl\{u\in L^\infty\, :\; \inf_{\bm h\in\Omega_\rho}\|u-u^{\bm h}\|_{{}_{L^\infty}}<\sigma\Bigr\},\\
	\hat{\mathcal{S}}_{\rho,\sigma}&:=\left\{(\bm h,w)\in\Omega_\rho\times L^\infty\,:\; \|w\|_{{}_{L^\infty}}<\sigma,\,
	\langle w,k^{\bm h}_j\rangle =0 \quad\mbox{ for }\; j=1,\dots,N\right\}.
	\end{aligned}
\end{equation*}
\vspace{0.2cm}
\begin{prop}[Carr--Pego \cite{Carr-Pego}] \label{tub-co}
There exist $\rho_1, \rho_2, \sigma, C>0$ with $\rho_1<\rho_2$
and a smooth function $\mathcal{H}\,:\,\mathcal{B}_{\rho_1,\sigma}\rightarrow\Omega_{\rho_2}$ 
such that, whenever $\bm h=\mathcal{H}(u)$, we have 
\begin{equation*}
	\langle u-u^{\bm h}, k^{\bm h}_j\rangle=0 \qquad \mbox{ for } \quad j=1,\dots,N,
\end{equation*}
and
\begin{equation*}
	\|u-u^{\bm h}\|_{{}_{L^\infty}}\leq C\inf\left\{\|u-u^{\bm l}\|_{{}_{L^\infty}}\,:\;\bm l\in\Omega_{\rho_1}\right\}<C\sigma.
\end{equation*}
Moreover, defining $\mathcal{U}\,:\,\hat{\mathcal{S}}_{\rho_1,\sigma}\to \mathcal{B}_{\rho_1,\sigma}$ by setting
\begin{equation*}
	\mathcal{U}(\bm h,w):=u^{\bm h}+w\quad\textrm{and}\quad
	\mathcal{S}_{\rho_1,\sigma}:=\mathcal{U}(\hat{\mathcal{S}}_{\rho_1,\sigma}),
\end{equation*}
the function $\mathcal{U}$ is injective, $(\mathcal{H}\circ\mathcal{U})(\bm h,w)=\bm h$
for all $(\bm h,w)\in\hat{\mathcal{S}}_{\rho_1,\sigma}$
and the set $\mathcal{S}_{\rho_1,\sigma}$ is open in $L^\infty(0,1)$. 
\end{prop}

In the last statement, constants $\rho_1, \rho_2, \sigma, C$ can be chosen independent on $\varepsilon$.

\subsection{Energy functional $\mathcal{E}^h$ and barrier function $\Psi$}
As stated in the Introduction, the neighborhood $\mathcal{Z}_{{}_{\Gamma,\rho}}$ of the extended base
manifold $\mathcal{M}_{{}_{0}}$ is defined in terms of the {\it energy functional} $\mathcal{E}^{\bm{h}}$
and the {\it barrier function} $\Psi$, see \eqref{energy} and \eqref{barrier}.
The positivity of the first term in $\mathcal{E}^{\bm{h}}$ holds for $\rho$ small and $w$ satisfying the
orthogonality condition \eqref{ortogonale}.
\vspace{0.2cm}
\begin{thm}[Carr--Pego \cite{Carr-Pego}] \label{L^hw-theo}
Let $f=F'$, with $F$ satisfying \eqref{hypf}.
Given $N\in\mathbb{N}$ and $\delta\in(0,1/N)$, there exist $\varepsilon_0, \Lambda>0$, 
such that if $\varepsilon$ and $\rho$ are chosen so that \eqref{triangle} holds and $\bm{h}\in\Omega_\rho$, then

\begin{equation*}
	\Lambda\int_0^1\bigl\{\varepsilon^2w^2_x+w^2\bigr\}dx
		\leq\int_0^1\bigl\{\varepsilon^2w_x^2+f'(u^{\bm h})w^2\bigr\}dx,
\end{equation*}
for any $w\in H^1(0,1)$ satisfying $\langle w,k^{\bm h}_j\rangle=0$ for $j=1,\dots,N$.
\end{thm}

Given ${\bm h}\in\Omega_\rho$, we consider the operator $L^{\bm h}$, linearization of $\mathcal{L}(u)$ about $u^{\bm h}$, i.e.
\begin{equation}\label{L^h}
	L^{\bm h}w:=-\varepsilon^2w_{xx}+f'(u^{\bm h})w.
\end{equation} 
If $w\in H^2$ and $w_x(0)=w_x(1)=0$, integrating by parts, we infer
\begin{equation*}
	\langle w,L^{\bm h}w\rangle=\int_0^1\bigl\{\varepsilon^2w^2_x+f'(u^{\bm h})w^2\bigr\}dx.
\end{equation*}
Hence, in this case, the energy functional can be written as
\begin{equation}\label{E(w,v)}
	\mathcal{E}^{\bm{h}}[w,v]=\tfrac12\langle w, L^{\bm h} w\rangle+\tfrac12\tau\|v\|^2+\varepsilon\tau\langle w,v\rangle,
\end{equation}
and from Theorem \ref{L^hw-theo} it follows that
\begin{equation}\label{eq:nuoval2}
	\Lambda\|w\|^2 \leq \langle w, L^{\bm h}w\rangle.
\end{equation}
Moreover, let $x_2\in[0,1]$ be such that $|w(x_2)|=\|w\|_{{}_{L^\infty}}$ 
and let $x_1\in[0,1]$ be such that $|w(x_1)|=\min\{|w(x)|:x\in[0,1]\}$.
Assume without loss of generality $x_2>x_1$ (otherwise replace $w(x)$ by $w(1-x)$).
We have
\begin{equation*}
	\varepsilon w(x_2)^2-\varepsilon w(x_1)^2= \int_{x_1}^{x_2} 2\varepsilon w(x)\,w_x(x)\,dx
	\leq \int_0^1\bigl\{\varepsilon^2w^2_x+w^2\bigr\}dx,
\end{equation*}
and so,
\begin{equation*}
	\varepsilon \|w\|_{{}_{L^\infty}}^2\leq \varepsilon w(x_1)^2+\int_0^1\bigl\{\varepsilon^2w^2_x+w^2\bigr\}dx\leq (1+\varepsilon)\int_0^1\bigl\{\varepsilon^2w^2_x+w^2\bigr\}.
\end{equation*}
By applying Theorem \ref{L^hw-theo} and taking into account the latter bound with $\varepsilon\leq1$,
we deduce also the estimate
\begin{equation}\label{intelligente}
	\tfrac12\Lambda\varepsilon\|w\|^2_{{}_{L^\infty}}\leq\int_0^1\bigl\{\varepsilon^2w_x^2+f'(u^{\bm h})w^2\bigr\}dx
		=\langle w,L^{\bm h}w\rangle.
\end{equation}
In order to provide representations of the barrier $\Psi$, defined in \eqref{barrier},
we introduce some auxiliary functions.
Since $\phi(0,\ell,\pm1)$ depends only on the ratio $r=\varepsilon/\ell$, we can define
\begin{equation*}
	\alpha_\pm(r):=F(\phi(0,\ell,\pm1)), \qquad \quad \beta_\pm(r):=1\mp\phi(0,\ell,\pm1).
\end{equation*}
By definition, $\phi(0,\ell,\pm1)$ is close to $+1$ or $-1$ and so, $\alpha_\pm(r), \beta_\pm(r)$ are close to $0$. 
The next result characterizes the leading terms in $\alpha_\pm$ and $\beta_\pm$ as $r\to 0$.
\vspace{0.2cm}
\begin{prop} [Carr--Pego \cite{Carr-Pego}] \label{prop:alfa,beta}
Let $F$ be such that \eqref{hypf} holds and set $A_\pm^2:=F''(\pm1)$.
There exists $r_0, K_\pm>0$ such that if $0<r<r_0$, then
\begin{equation*}
	\begin{aligned}
	\alpha_\pm(r)&=\tfrac12K^2_\pm A^2_\pm\,\exp(-{A_\pm}/r\bigr)\bigl\{1+O\left(r^{-1} \exp(-{A_\pm}/2r)\right)\bigr\},\\
	\beta_\pm(r)&=K_\pm\,\exp\bigl(-{A_\pm}/2r\bigr)\bigl\{1+O\left(r^{-1} \exp(-{A_\pm}/2r)\right)\bigr\},
	\end{aligned}
\end{equation*}
with corresponding asymptotic formulae for the derivatives of $\alpha_\pm$ and $\beta_\pm$.
\end{prop}

\noindent
Explicit expressions of $K_\pm$ in terms of $F$ can be found in \cite{Carr-Pego}.

For $j=0,\dots,N$, we set
\begin{equation*}
	r_{j+1/2}:=\frac{\varepsilon}{h_{j+1}-h_{j}},
\end{equation*}
and
\begin{equation*}
	\alpha^{j+1/2}:=\left\{\begin{aligned}
		&\alpha_+(r_{j+1/2}) 	&j \textrm{ odd},\\
		&\alpha_-(r_{j+1/2})  	&j \textrm{ even},\\
		\end{aligned}\right.
	\qquad
	\beta^{j+1/2}:=\left\{\begin{aligned}
		&\beta_+(r_{j+1/2})	&j \textrm{ odd},\\
		&\beta_-(r_{j+1/2})	&j \textrm{ even},\\
		\end{aligned}\right.
\end{equation*}
For $\bm h\in\Omega_\rho$, since
\begin{equation*}
	\begin{aligned}
	&\mathcal{L}(u^{\bm h}(x))=0		&\textrm{if}\; |x-h_j|\geq\varepsilon,\\
	&k^{\bm h}_j(x)=-u^{\bm h}_x(x) 	&\textrm{if}\; |x-h_j|\leq\varepsilon,
	\end{aligned}
\end{equation*}
direct integration gives
\begin{equation*}
	\bigl\langle\mathcal{L}(u^{\bm h}),k^{\bm h}_j\bigr\rangle
		=\int_{h_j-\varepsilon}^{h_j+\varepsilon}\bigl\{\varepsilon^2u^{\bm h}_{xx}-f(u^{\bm h})\bigr\}u^{\bm h}_x\,dx
		=\alpha^{j-1/2}-\alpha^{j+1/2}.
\end{equation*}
Thus, the barrier function $\Psi$, defined in \eqref{barrier}, can be written as
\begin{equation}\label{Psi(h)}
	\Psi(\bm h)=\sum_{j=1}^N\bigl(\alpha^{j-1/2}-\alpha^{j+1/2}\bigr)^2.
\end{equation}
The next statement collects some estimates we will use later on.
\vspace{0.2cm}
\begin{prop}\label{prop-L(u^h)}
Let $f=F'$, with $F$ satisfying \eqref{hypf}.
Given $N\in\mathbb{N}$ and $\delta\in(0,1/N)$, there exist $\varepsilon_0, C>0$, 
such that if $\varepsilon$ and $\rho$ are chosen so that \eqref{triangle} holds and $\bm{h}\in\Omega_\rho$, then
\begin{align}
	\|\mathcal{L}(u^{\bm h})\|&\leq C\varepsilon^{1/2}\sum_{j=1}^N\bigl|\alpha^{j+1/2}-\alpha^{j-1/2}\bigr|
		\leq C\varepsilon^{1/2}\exp(-A\ell^{\bm h}/\varepsilon), \label{||L(u^h)||}\\
	\|L^{\bm h}u^{\bm h}_j\|&\leq C\varepsilon^{-1/2}\max\{\alpha^{j-1/2},\alpha^{j+1/2}\}
		\leq C\varepsilon^{-1/2}\exp(-A\ell^{\bm h}/\varepsilon), \label{|dhL(u^h)|}
\end{align}
where $\ell^{\bm h}:=\min\{h_j-h_{j-1}\,:\,j=1,\dots,N+1\}$ and $A:=\min A_\pm$.
\end{prop}

Here, we give only an idea of the proofs, the complete ones can be found
in \cite{Carr-Pego} for \eqref{||L(u^h)||} and in \cite{Carr-Pego2} for \eqref{|dhL(u^h)|}.
Recalling the definition \eqref{u^h(x)}, for $x\in I_j$, we have
\begin{equation}\label{L(u^h)}
	\mathcal{L}(u^{\bm h})=\varepsilon^2\chi^j_{xx}\left(\phi^{j}-\phi^{j+1}\right)
		+2\varepsilon^2\chi^j_x\left(\phi^{j}_x-\phi^{j+1}_x \right)-G,
\end{equation}
where the remainder $G$ is given by
\begin{equation*}
	G=\left(1-\chi^j\right)f(\phi^j)+\chi^jf(\phi^{j+1})-f\left((1-\chi^j)\phi^j+\chi^j\phi^{j+1}\right).
\end{equation*}
Using Lagrange interpolation formula,
\begin{equation*}
	G=\left(\phi^{j+1}-\phi^j\right)^2\biggl\{(1-\chi^j)\int_0^{\chi^j}sf''(\theta)ds+\chi^j\int_{\chi^j}^1(1-s)f''(\theta)ds\biggr\},
\end{equation*}
with $\theta(s)=(1-s)\phi^j(x)+s\phi^{j+1}(x)$. It can be shown (see \cite[Lemma 8.2]{Carr-Pego}) that there exist $C>0$
such that for $x\in[h_j-\varepsilon,h_j+\varepsilon]$, $j\in\{1,\dots,N\}$, we have
\begin{equation}\label{fi^j-fi^j+1}
	\left|\phi^j(x)-\phi^{j+1}(x)\right|+\varepsilon\left|\phi^j_x(x)-\phi^{j+1}_x(x)\right|
	<C\left|\alpha^{j-1/2}-\alpha^{j+1/2}\right|,
\end{equation}
provided $r_j, r_{j+1}<r_0$ and $r_0$ is sufficiently small.
Using these estimates, the fact that $\mathcal{L}(u^{\bm h}(x))=0$ if $|x-h_j|>\varepsilon$ and that
the $m$-th derivative of $\varepsilon^{m}\chi^{j}$ is uniformly bounded (independently on $\varepsilon$),
we obtain
\begin{equation*}
	|\mathcal{L}(u^{\bm h}(x))|\leq C\left|\alpha^{j+1/2}-\alpha^{j-1/2}\right| \quad \textrm{for }x\in I_j.
\end{equation*}
Then, the $L^2$-bound \eqref{||L(u^h)||} follows since $\mathcal{L}(u^{\bm h})$
has support of length $2\varepsilon$ in $I_j$. 
The estimate \eqref{|dhL(u^h)|} is obtained in a similar way, by differentiating \eqref{L(u^h)} with respect to $\bm h$.

\section{Dynamics near the base manifold}\label{motion}
In this section we study the dynamics of \eqref{system-u-v}-\eqref{Neumann} in a neighborhood of $\mathcal{M}_{{}_{0}}$
using the decomposition $u=u^{\bm h}+w$ and deriving the system of equations for $(\bm h,w,v)$.
Such description will be used to prove Theorem~\ref{main}. 

\subsection{Equations for the motion}
Let $(u,v)$ be a classical solution of \eqref{system-u-v}-\eqref{Neumann}, 
with $u$ lying in the tubular neighborhood $\mathcal{S}_{\rho,\sigma}$ for $t\in[0,T]$ with $T>0$. 
Let $\bm h(t)=\mathcal{H}(u(\cdot,t))$ and $w(x,t)=u(x,t)-u^{\bm h(t)}$, where $u^{\bm h}$ is defined by \eqref{u^h(x)}. 
We recall that $u(\cdot,t)\in\mathcal{S}_{\rho,\sigma}$ for $t\in[0,T]$ means that $\bm h(t)\in\Omega_\rho$, $w(\cdot,t)\in H^2_N$
and $\|w(\cdot,t)\|_{{}_{L^\infty}}<\sigma$ for $t\in[0,T]$. 
Moreover, $v(\cdot,t)\in L^2(0,1)$ for $t\in[0,T]$.

From \eqref{system-u-v} it follows that the pair $(w,v)$ satisfies
\begin{equation}\label{system-w-v}
	\begin{cases}
		w_t=v-\nabla_{\bm h} u^{\bm h}\cdot {\bm h}',\\
		\tau v_t=-\mathcal{L}(u^{\bm h}+w)-g(u^{\bm h}+w,\tau)v,
	\end{cases}
\end{equation}
where $\,\cdot\,$ denotes the inner product in $\mathbb{R}^N$. 
Expanding, we get
\begin{equation*}
	\mathcal{L}(u^{\bm h}+w)=\mathcal{L}(u^{\bm h})+L^{\bm h}w-f_2w^2,
	\qquad\textrm{where}\quad
	f_2:=\int_0^1(1-s)f''(u^{\bm h}+sw)\,ds,
\end{equation*}
and $L^{\bm h}$ is the differential operator defined in \eqref{L^h}.

Differentiating with respect to $t$ the orthogonality condition \eqref{ortogonale}, we obtain
\begin{equation}\label{h'}
	\sum_{j=1}^N\bigl\{\langle u^{\bm h}_j,k^{\bm h}_i\rangle-\langle w,k^{\bm h}_{ij}\rangle\bigr\}{\bm h}'_j
	=\langle v,k^{\bm h}_i\rangle, \qquad i=1,\dots,N.
\end{equation}
Using the notation
\begin{equation*}
	D_{ij}(\bm h):=\langle u^{\bm h}_j,k^{\bm h}_i\rangle, \qquad
	\hat{D}_{ij}(\bm h,w):=\langle w,k^{\bm h}_{ij}\rangle, \qquad
	Y_i(\bm h,v):=\langle v,k^{\bm h}_i\rangle,
\end{equation*}
equation \eqref{h'} becomes
\begin{equation}\label{h'-compact}
	\bigl\{D(\bm h)-\hat{D}(\bm h,w)\bigr\}\bm h'=Y(\bm h,v).
\end{equation}
From \eqref{system-w-v} and \eqref{h'-compact}, we obtain the ODE-PDE coupled system
\begin{equation*}
	\begin{cases}
		w_t=v-\nabla_{\bm h} u^{\bm h}\cdot\bm h',\\
		\tau v_t=-\mathcal{L}(u^{\bm h})-L^{\bm h}w+f_2w^2-g(u^{\bm h}+w,\tau)v,\\
		\bigl\{D(\bm h)-\hat{D}(\bm h,w)\bigr\}\bm h'=Y(\bm h,v).
	\end{cases}
\end{equation*}
The matrix $D(\bm h)$ is diagonally dominant, because, for any $\eta\in(0,1)$ there exists $\rho_0>0$ such that if 
$\rho<\rho_0$, then
\begin{align}
	D_{ii}(\bm h)-\sum_{j\neq i}|D_{ij}(\bm h)|&=\langle u^{\bm h}_i,k^{\bm h}_i\rangle-\sum_{j\neq i}|\langle u^{\bm h}_j,k^{\bm h}_i\rangle| \notag\\
	&\geq \bigl\{\bigl(A_0-\omega(\rho)\bigr)^2-(N-1)\omega(\rho)\bigr\}\varepsilon^{-1}>\eta A^2_0\,\varepsilon^{-1}, \label{eq:D}
\end{align}
thanks to Proposition \ref{estimates-u^h_j}.
Also, for a known property of inverses of diagonally dominant matrices (see \cite{Varah}), 
$D(\bm h)$ is invertible and it holds
\begin{equation*}
	\|D^{-1}(\bm h)\|_{{}_{\infty}}\leq \eta^{-1}A_0^{-2}\,\varepsilon,
\end{equation*} 
where $\|\cdot\|_{{}_{\infty}}$ denotes the operator norm induced by the norm $|\cdot|_{{}_{\infty}}$.
In Section \ref{layer}, it is determined the explicit expression for the principal term in the expansion
of the inverse $D^{-1}(\bm h)$ as $\varepsilon\to 0$.

The invertibility of the matrix $D(\bm h)-\hat{D}(\bm h,w)$ descends from the smallness of $\hat{D}(\bm h,w)$ for $w\to 0$ and \eqref{eq:D}.
Indeed, for $(\bm h,w)\in\hat{\mathcal{S}}_{\rho,\sigma}$, applying Proposition  \ref{estimates-u^h_j}, we infer
\begin{equation*}
	\begin{aligned}
	\sum_j |\hat D_{ij}(\bm h,w)|
	&\leq \|w\|_{{}_{L^\infty}} \Bigl\{\|k^{\bm h}_{ii}\|_{{}_{L^1}}+\sum_{j\neq i} \|k^{\bm h}_{ij}\|_{{}_{L^1}}\Bigr\}\\
	&\leq \sigma \bigl\{C+(N-1)\omega(\rho)\bigr\}\varepsilon^{-1},
	\end{aligned}
\end{equation*}
and thus 
\begin{equation*}
	D_{ii}(\bm h)-\hat D_{ii}(\bm h,w)-\sum_{j\neq i}|D_{ij}(\bm h)-\hat D_{ij}(\bm h,w)|
	\geq\eta\,A^2_0\,\varepsilon^{-1}-\sigma \bigl\{C+(N-1)\omega(\rho)\bigr\}\varepsilon^{-1}.
\end{equation*}
Therefore, the matrix $D-\hat{D}$ is invertible for $\rho$ and $\sigma$ sufficiently small and
\begin{equation}\label{boundDinverse}
	\|\bigl\{D(\bm h)-\hat{D}(\bm h,w)\bigr\}^{-1}\|_{{}_{\infty}}\leq 2\eta\,A_0^{-2}\,\varepsilon.
\end{equation} 
Applying $\bigl\{D(\bm h)-\hat{D}(\bm h,w)\bigr\}^{-1}$ in the equation for $\bm{h}$,
we obtain the final form of the system
\begin{equation}\label{system-w-v-h}
	\begin{cases}
		w_t=v-\nabla_{\bm h} u^{\bm h}\cdot\bm h',\\
		\tau v_t=-\mathcal{L}(u^{\bm h})-L^{\bm h}w+f_2w^2-g(u^{\bm h}+w,\tau)v,\\
		\bm h'=\bigl\{D(\bm h)-\hat{D}(\bm h,w)\bigr\}^{-1}Y(\bm h,v).
	\end{cases}
\end{equation}
The proof of our main result consists in providing estimates for the solutions to \eqref{system-w-v-h}.

\subsection{Proof of the main result}
To start with, we observe that if  $(\bm h,w)\in\hat{\mathcal{S}}_{\rho,\sigma}$ for $\rho,\sigma$ small
then there exists $C>0$ such that
\begin{equation}\label{|h'|<}
	|\bm h'|_{{}_{\infty}}\leq 2\eta A_0^{-2}\varepsilon |Y(\bm h,v)|_{{}_{\infty}}\leq C\varepsilon^{1/2}\|v\|,
\end{equation}
using \eqref{boundDinverse} and the third estimate in Proposition \ref{estimates-u^h_j}.

In order to prove Theorem \ref{main}, we restrict the attention to the set
\begin{equation*}
	\hat{\mathcal{Z}}_{{}_{\Gamma,\rho}}:=\bigl\{(\bm h,w,v)\in\overline{\Omega}_\rho\times H^2_N\times L^2(0,1)
		\, : \, \mathcal{E}^{\bm{h}}[w,v]\leq\Gamma\Psi(\bm h)\bigr\}.
\end{equation*}
The aim of the next result is twofold.
Firstly, it states that if the triple $(\bm h,w,v)$ belongs to $\hat{\mathcal{Z}}_{{}_{\Gamma,\rho}}$ then
the bound on $(w,v)$  stated in Theorem \ref{main}, estimate \eqref{umenouh}, holds true.
Secondly, assuming in addition that $(\bm h,w,v)$ is a solution to \eqref{system-w-v-h},
then also the bound on $\bm{h}'$ in Theorem \ref{main}, estimate \eqref{|h'|<exp-intro}, is valid.
\vspace{0.2cm}
\begin{prop}\label{prop:E>}
Let $F\in C^3$ be such that \eqref{hypf} holds and $g(\cdot,\tau)\in C^1$.
Given $N\in\mathbb{N}$ and $\delta\in(0,1/N)$, there exist $\varepsilon_0, C>0$, 
such that for $\varepsilon$ and $\rho$ satisfying \eqref{triangle},
\begin{itemize}
\item[(i)] if $(\bm h,w,v)\in\hat{\mathcal{Z}}_{{}_{\Gamma,\rho}}$, then
\begin{equation}\label{E>}
   \begin{aligned}
	\tfrac18\Lambda\varepsilon\|w\|^2_{{}_{L^\infty}}+ \tfrac14\tau\|v\|^2\leq \mathcal{E}^{\bm{h}}[w,v],
	\\
	 \tfrac14 \Lambda \|w\|^2 + \tfrac14\tau\|v\|^2\leq \mathcal{E}^{\bm{h}}[w,v],
	 \\
	 \mathcal{E}^{\bm{h}}[w,v]\leq C\Gamma\exp(-2A\ell^{\bm h}/\varepsilon),
  \end{aligned}
\end{equation}
where $\Lambda$ is the positive constant introduced in Theorem \ref{L^hw-theo};
\item[(ii)] if $(\bm h,w,v)\in\hat{\mathcal{Z}}_{{}_{\Gamma,\rho}}$ is a solution of \eqref{system-w-v-h} for $t\in[0,T]$, then
\begin{equation}\label{|h'|<exp}
	|\bm h'|_{{}_{\infty}}\leq C(\varepsilon/\tau)^{1/2}\exp(-A\ell^{\bm h}/\varepsilon).
\end{equation}
\end{itemize}
\end{prop}
\begin{proof}
Let  us prove the first inequality in \eqref{E>}.
Using Young inequality, we have
\begin{equation*}
	\varepsilon|\langle w,v\rangle|\leq\varepsilon^2\|w\|^2+\tfrac14\|v\|^2
	\leq\varepsilon^2\|w\|_{{}_{L^\infty}}^2+\tfrac14\|v\|^2,
\end{equation*}
and so, recalling the expression for the energy $\mathcal{E}^{\bm{h}}$ given in \eqref{E(w,v)},
\begin{equation*}
	\mathcal{E}^{\bm{h}}[w,v]\geq \tfrac12\langle w, L^{\bm h}w\rangle+\tfrac14\tau\|v\|^2
		-\varepsilon^2\tau\|w\|_{{}_{L^\infty}}^2.
\end{equation*}
Using \eqref{intelligente}, we obtain, for $\varepsilon<\Lambda/8\tau$,
\begin{equation*}
	\mathcal{E}^{\bm{h}}[w,v]\geq\bigl(\tfrac14\Lambda-\varepsilon\tau\bigr) \varepsilon\|w\|^2_{{}_{L^\infty}}
		+\tfrac14\tau\|v\|^2\geq \tfrac18\Lambda\varepsilon\|w\|^2_{{}_{L^\infty}}+\tfrac14\tau\|v\|^2.
\end{equation*}
Moreover, from \eqref{eq:nuoval2}, for $\varepsilon^2<\Lambda/4\tau$ one has
\begin{equation*}
	\mathcal{E}^{\bm{h}}[w,v]\geq\bigl(\tfrac12\Lambda-\varepsilon^2\tau\bigr) \|w\|^2
		+\tfrac14\tau\|v\|^2\geq \tfrac14\Lambda\|w\|^2+\tfrac14\tau\|v\|^2,
\end{equation*}
concluding the first two inequalities of \eqref{E>}.
The upper bound for $\mathcal{E}^{\bm{h}}[w,v]$ follows from the definition of $\hat{\mathcal{Z}}_{{}_{\Gamma,\rho}}$, 
the expression of the barrier $\Psi$ given in \eqref{Psi(h)} and Proposition~\ref{prop:alfa,beta}.

To prove part (ii), using \eqref{E>}, we deduce the estimate
\begin{equation*}
	\|w\|_{{}_{L^\infty}}\leq C\varepsilon^{-1/2}\exp(-A\ell^{\bm h}/\varepsilon)
		\leq C\varepsilon^{-1/2}\exp(-A\delta/\varepsilon)=:\sigma
\end{equation*}
since, by definition and \eqref{triangle}, $\ell^{\bm h}>\varepsilon/\rho>\delta$.
Hence, for $\varepsilon$ sufficiently small, both $\rho$ and $\sigma$ are small, and thus,
for $(\bm h,w)\in\hat{\mathcal{S}}_{\rho,\sigma}$, estimate \eqref{|h'|<} holds. 
Therefore, if $(\bm h,w,v)\in\hat{\mathcal{Z}}_{{}_{\Gamma,\rho}}$, then $(\bm h,w)\in\hat{\mathcal{S}}_{\rho,\sigma}$
and the estimate \eqref{|h'|<exp} is obtained by applying \eqref{E>} in \eqref{|h'|<}.
\end{proof}

Now, we estimate the time $T$ taken for the solution $(u,v)$ to leave $\mathcal{Z}_{{}_{\Gamma,\rho}}$.
To do this, we study the system \eqref{system-w-v-h} in the set $\hat{\mathcal{Z}}_{{}_{\Gamma,\rho}}$
by using energy estimates.
\vspace{0.2cm}
\begin{prop}\label{prop:d/dtE}
Let $F\in C^3$ and $g(\cdot,\tau)\in C^1$ be such that \eqref{hypf} and \eqref{hypg} hold.
Given $N\in\mathbb{N}$ and $\delta\in(0,1/N)$, there exist $\Gamma_2>\Gamma_1>0$ and $\varepsilon_0>0$ 
such that if $\Gamma\in[\Gamma_1,\Gamma_2]$, $\varepsilon,\rho$ satisfy \eqref{triangle} 
and $(\bm h,w,v)\in\hat{\mathcal{Z}}_{{}_{\Gamma,\rho}}$ is a solution of \eqref{system-w-v-h} for $t\in[0,T]$, then
for some $\eta\in(0,1)$, we have
\begin{equation}\label{E-GPsi^2}
	\frac{d}{dt}\bigl\{\mathcal{E}^{\bm{h}}[w,v]-\Gamma\Psi(\bm h)\bigr\}
		\leq-\eta\,\varepsilon\bigl\{\mathcal{E}^{\bm{h}}[w,v]-\Gamma\Psi(\bm h)\bigr\}
	\qquad\textrm{for}\quad t\in[0,T].
\end{equation}
\end{prop}

\begin{proof}
In all the proof, symbols $C, c, \eta$ denote generic positive constants, independent on $\varepsilon$,
and with $\eta\in(0,1)$.
Let us recall that, if $(\bm{h},w,v)$ is a solution to \eqref{system-w-v-h}, then
\begin{equation*}
	w_t=v-\nabla_{\bm h} u^{\bm h}\cdot\bm h',\qquad
	\tau v_t=-\mathcal{L}(u^{\bm h})-L^{\bm h}w+f_2w^2-g(u^{\bm h}+w,\tau)v.
\end{equation*}
Direct differentiation and the self-adjointness of the operator $L^{\bm{h}}$ give
\begin{equation*}
	\begin{aligned}
	\frac{d}{dt}\Bigl\{\tfrac12\langle w,L^{\bm h}w\rangle\Bigr\}
	& =\langle w_t,L^{\bm h}w\rangle+\tfrac12\langle w,f''(u^{\bm h})\bigl(\nabla_{\bm h}u^{\bm h}\cdot\bm h'\bigr)w\rangle\\
	& =\langle v,L^{\bm h}w\rangle-\langle \nabla_{\bm h} u^{\bm h}\cdot\bm h',L^{\bm h}w\rangle
		+\tfrac12\langle w,f''(u^{\bm h})\bigl(\nabla_{\bm h}u^{\bm h}\cdot\bm h'\bigr)w\rangle\\
	&= \langle v,L^{\bm h}w\rangle-\langle L^{\bm h} \nabla_{\bm h} u^{\bm h}\cdot\bm h',w\rangle
		+\tfrac12\langle w,f''(u^{\bm h})\bigl(\nabla_{\bm h}u^{\bm h}\cdot\bm h'\bigr)w\rangle.
	\end{aligned}
\end{equation*}
Using  Cauchy--Schwarz inequality and the estimates in Proposition \ref{prop-L(u^h)} and  \ref{estimates-u^h_j}, we infer
\begin{equation*}
	\begin{aligned}
	\frac{d}{dt}\Bigl\{\tfrac12\langle w,L^{\bm h}w\rangle\Bigr\}
	&\leq \langle v,L^{\bm h}w\rangle+C\sum_j\Bigl(\|L^{\bm h} u^{\bm h}_j\|
		+\|u^{\bm h}_j\|_{{}_{L^\infty}}\|w\|\Bigr)|\bm{h}'|_{{}_{\infty}}\|w\|\\
	&\leq \langle v,L^{\bm h}w\rangle+C\varepsilon^{-1/2}\bigl(\exp(-A\delta/\varepsilon)
		+\varepsilon^{-1/2}\|w\| \bigr)|\bm{h}'|_{{}_{\infty}}\|w\|.
	\end{aligned}
\end{equation*}
For $(\bm{h},w,v)\in\hat{\mathcal{Z}}_{{}_{\Gamma,\rho}}$, applying \eqref{|h'|<}, \eqref{E>} and using Young inequality, we get 
\begin{equation*}
	\begin{aligned}
	\frac{d}{dt}\Bigl\{\tfrac12\langle w,L^{\bm h}w\rangle\Bigr\}
	&\leq \langle v,L^{\bm h}w\rangle+C\bigl(\exp(-A\delta/\varepsilon)
		+\varepsilon^{-1/2}\|w\| \bigr)\|w\|\|v\|\\
	&\leq \langle v,L^{\bm h}w\rangle+C\exp(-2A\delta/\varepsilon)\bigl( 1 + \varepsilon^{-1}\Gamma\bigr)\|w\|^2+\eta\|v\|^2.
	\end{aligned}
\end{equation*}
For what concerns the second term in the energy $\mathcal{E}^{\bm{h}}$, it holds
\begin{equation*}
	\begin{aligned}
	\frac{d}{dt}\Bigl\{\tfrac12\tau\|v\|^2\Bigr\}& =\langle \tau v_t,v\rangle
		=\langle -\mathcal{L}(u^{\bm h})-L^{\bm h}w+f_2w^2-g(u^{\bm h}+w,\tau)v,v\rangle\\
	&\leq-\langle L^{\bm h}w,v\rangle+\|\mathcal{L}(u^{\bm h})\|\|v\|+C\|w\|_{{}_{L^\infty}}\|w\|\|v\|-c_g\|v\|^2\\
	&\leq -\langle L^{\bm h}w,v\rangle+C\|w\|_{{}_{L^\infty}}^2\|w\|^2-(c_g-\eta)\|v\|^2
		+C\|\mathcal{L}(u^{\bm h})\|^2.
	\end{aligned}
\end{equation*}
Finally, the time derivative of the scalar product $\langle w,\tau v\rangle$ can be bounded by
\begin{equation*}
	\begin{aligned}
	\frac{d}{dt}\langle w,\tau v\rangle 
	&=\langle v-\nabla_{\bm h} u^{\bm h}\cdot\bm h',\tau v\rangle+\langle w,-\mathcal{L}(u^{\bm h})-L^{\bm h}w
		+f_2w^2-g(u^{\bm h}+w,\tau)v\rangle\\
	&\leq -\langle w,L^{\bm h}w\rangle+C(\varepsilon+\|w\|_{{}_{L^\infty}})\|w\|^2
		 +(\tau+\eta\,\varepsilon^{-1})\|v\|^2+C\tau\varepsilon^{-1/2}|\bm h'|_{{}_{L^\infty}}\|v\|\\
	&\hskip9cm +\varepsilon^{-1}\|\mathcal{L}(u^{\bm h})\|^2\\
	&\leq -\langle w,L^{\bm h}w\rangle+C(\varepsilon+\|w\|_{{}_{L^\infty}})\|w\|^2
		+(C+\eta\,\varepsilon^{-1})\|v\|^2+\varepsilon^{-1}\|\mathcal{L}(u^{\bm h})\|^2,
	\end{aligned}
\end{equation*}
where, in particular, the inequalities
\begin{equation*}
	\begin{aligned}
	\langle w,\mathcal{L}(u^{\bm h})\rangle&\leq \tfrac12\varepsilon\|w\|^2+\tfrac12\varepsilon^{-1}\|\mathcal{L}(u^{\bm h})\|^2,\\
	\langle w,g(u^{\bm h}+w,\tau)v\rangle&\leq C\varepsilon\|w\|^2+\eta\,\varepsilon^{-1}\|v\|^2
	\end{aligned}
\end{equation*}
have been used.
Collecting the estimates for the three terms composing $\mathcal{E}^{\bm h}$, we deduce
\begin{equation*}
	\begin{aligned}
	\frac{d\mathcal{E}^{\bm h}}{dt} &\leq -\varepsilon\langle w,L^{\bm h}w\rangle
		-[c_g-C\varepsilon-3\eta]\|v\|^2\\
	&\hskip1.0cm  +C\bigl\{ \exp(-2A\delta/\varepsilon)\bigl( 1 + \varepsilon^{-1}\Gamma\bigr)
		+\varepsilon (\varepsilon+\|w\|_{{}_{L^\infty}})\bigr\}\|w\|^2+(C+1)\|\mathcal{L}(u^{\bm h})\|^2\\
	&\leq -\varepsilon\langle w,L^{\bm h}w\rangle+C\varepsilon\bigl\{
		\Gamma\exp(-c/\varepsilon)+\varepsilon\bigr\}\|w\|^2-\eta c_g\|v\|^2+C\|\mathcal{L}(u^{\bm h})\|^2,
	\end{aligned}
\end{equation*}
for $\varepsilon$ and $\eta$ small.
Thus, from \eqref{eq:nuoval2} and 
\begin{equation}\label{L(u^h)<Psi}
	\|\mathcal{L}(u^{\bm h})\|^2\leq C\varepsilon\,\Psi(\bm h)\leq C\varepsilon\exp(-2A\ell^{\bm h}/\varepsilon),
\end{equation}
it follows that
\begin{equation*}
	\frac{d\mathcal{E}^{\bm h}}{dt} \leq -\varepsilon
	\bigl\{1-C\bigl(\Gamma \exp(-c/\varepsilon)+\varepsilon\bigr)\bigr\}\langle w,L^{\bm h}w\rangle
		-\eta\,c_g\|v\|^2+C\varepsilon\Psi.
\end{equation*} 
Hence, for $\varepsilon\in(0,\varepsilon_0)$, with $\varepsilon_0$ small (and dependent on $\Gamma$), we deduce the bound
\begin{equation*}
	1- C\bigl(\Gamma\exp(-c/\varepsilon)+\varepsilon\bigr)
	\geq \eta.
\end{equation*}
Substituting, we infer
\begin{equation*}
	\begin{aligned}
	\frac{d\mathcal{E}^{\bm h}}{dt} &\leq -\eta\,\varepsilon\langle w,L^{\bm h}w\rangle
	-\eta\,c_g\|v\|^2+C\varepsilon\Psi\\
	&\leq -\eta\,\varepsilon\mathcal{E}^{\bm h}
		-\tfrac12\eta\,\varepsilon\langle w,L^{\bm h}w\rangle+\eta\,\varepsilon^2\tau\langle w,v\rangle
		-\eta\bigl(c_g-\tfrac12\varepsilon\tau\bigr)\|v\|^2+C\varepsilon\Psi\\
	&\leq -\eta\,\varepsilon\mathcal{E}^{\bm h}
		-\tfrac12\eta\,\varepsilon\bigl(1-C\varepsilon\tau\bigr)\langle w,L^{\bm h}w\rangle
		-\eta\bigl(c_g-C\varepsilon\tau\bigr)\|v\|^2+C\varepsilon\Psi,
	\end{aligned}
\end{equation*} 
again from \eqref{eq:nuoval2}.
Finally, for $\varepsilon_0$ sufficiently small, we obtain
\begin{equation}\label{eq:E'}
	\frac{d\mathcal{E}^{\bm h}}{dt} \leq -\eta\,\varepsilon\mathcal{E}^{\bm h}
		-\eta\,c_g\|v\|^2+C\varepsilon\Psi.
\end{equation} 
Direct differentiation gives
\begin{equation*}
	\frac{d\Psi}{dt}=2\sum_{j=1}^N\langle\mathcal{L}(u^{\bm h}),k^{\bm h}_j\rangle
		\Bigl\{\langle\mathcal{L}(u^{\bm h}),\nabla_{\bm h}k_j^{\bm h}\cdot\bm h'\rangle
		-\langle L^{\bm h}\nabla_{{_{\bm h}}}u^{\bm h}\cdot\bm h',k^{\bm h}_j\rangle\Bigr\}.
\end{equation*}
Using the estimates provided by Proposition \ref{estimates-u^h_j} and by \eqref{|dhL(u^h)|}, \eqref{|h'|<}, we deduce
\begin{equation*}
	\begin{aligned}
	\bigl|\langle\mathcal{L}(u^{\bm h}),\nabla_{\bm h}k^{\bm h}_j\cdot\bm h'\rangle\bigr|
	&\leq|\bm h'|_{{}_\infty}\|\mathcal{L}(u^{\bm h})\|\sum_{i=1}^N\|k^{\bm h}_{ji}\|
		\leq C\varepsilon^{-1}\|\mathcal{L}(u^{\bm h})\|\|v\|,\\
	\bigl|\langle L^{\bm h}\nabla_{\bm h}u^{\bm h}\cdot\bm h',k^{\bm h}_j\rangle\bigr|
	&\leq|\bm h'|_{{}_\infty}\|k^{\bm h}_j\|\sum_{i=1}^N\|L^{\bm h}u^{\bm h}_i\|
		\leq C\exp(-c/\varepsilon)\|v\|,
	\end{aligned}
\end{equation*}
thus, observing that $|\langle\mathcal{L}(u^{\bm h}),k^{\bm h}_j\rangle|\leq C\varepsilon^{-1/2} \|\mathcal{L}(u^{\bm h})\|$,
we infer the bound
\begin{equation*}
	\left|\frac{d\Psi}{dt}\right|
	\leq C\varepsilon^{-1/2} \left\{\varepsilon^{-1}\|\mathcal{L}(u^{\bm h})\|+\exp(-c/\varepsilon)\right\}
		\|\mathcal{L}(u^{\bm h})\|\|v\|.
\end{equation*}
Using the inequality \eqref{L(u^h)<Psi}, we obtain
\begin{equation*}
	\begin{aligned}
	\left|\Gamma\frac{d\Psi}{dt}\right|
	&\leq C\,\Gamma\,\varepsilon^{-1/2}\bigl\{\Psi^{1/2}+\exp(-c/\varepsilon)\bigr\}\|v\|\Psi^{1/2}\\
	&\leq \eta\|v\|^2+C\,\Gamma^2\varepsilon^{-1}\bigl\{\Psi^{1/2}+\exp(-c/\varepsilon)\bigr\}^2\Psi.
	\end{aligned}
\end{equation*}
Hence, observing that $\Psi\leq C\exp\bigl(-c/\varepsilon\bigr)$, we end up with
\begin{equation}\label{eq:Psi'}
	\left|\Gamma\frac{d\Psi}{dt}\right|\leq \eta\|v\|^2+C\,\Gamma^2\exp(-c/\varepsilon)\Psi.
\end{equation}
Combining \eqref{eq:E'} and \eqref{eq:Psi'}, we obtain that if $(\bm h,w,v)\in\hat{\mathcal{Z}}_{{}_{\Gamma,\rho}}$ is a solution of \eqref{system-w-v-h}, then
\begin{equation*}
	\frac d{dt}\bigl\{\mathcal{E}^{\bm{h}}[w,v]-\Gamma\Psi(\bm h)\bigr\} \leq
	-\eta\,\varepsilon\mathcal{E}^{\bm h}+C\bigl(\varepsilon+\Gamma^2\exp(-c/\varepsilon)\bigr)\Psi,
\end{equation*}
for some $\eta\in(0,1)$.
Therefore the estimate \eqref{E-GPsi^2} follows from 
\begin{equation*}
C\exp(-c/\varepsilon)\Gamma^2-\eta\,\varepsilon\Gamma +C\varepsilon\leq 0,
\end{equation*}
 and the latter is verified for $\Gamma\in [\Gamma_1,\Gamma_2]$, provided $\varepsilon\in(0,\varepsilon_0)$ with $\varepsilon_0$ sufficiently small so that $\eta^2\varepsilon - 4C^2\exp(-c/\varepsilon) > 0$.
%
%
%
%
%
%
%
%
%
%
%
%
%
\end{proof}

Now, we have all the tools needed to prove Theorem \ref{main}.

\begin{proof}[Proof of Theorem \ref{main}]
Let $(u_0,v_0)\in\,\stackrel{\circ}{\mathcal{Z}}_{{}_{\Gamma,\rho}}$ and 
let $(u,v)$ be the solution of \eqref{system-u-v}-\eqref{Neumann}-\eqref{initial}. 
Assume that $(u,v)\in\mathcal{Z}_{{}_{\Gamma,\rho}}$ for $t\in[0,T_\varepsilon]$, where $T_\varepsilon$ is maximal.
Then, $u=u^{\bm h}+w$ and $(\bm h,w,v)\in\hat{\mathcal{Z}}_{{}_{\Gamma,\rho}}$ solves the system \eqref{system-w-v-h} for $t\in[0,T_\varepsilon]$. 
Let us apply Proposition \ref{prop:d/dtE}; from \eqref{E-GPsi^2}, it follows that
\begin{equation*}
	\frac d{dt}\Bigl\{\exp(\eta\,\varepsilon t)(\mathcal{E}^{\bm{h}}[w,v]-\Gamma\Psi(\bm h))\Bigr\}\leq0,
	\quad \qquad t\in[0,T_\varepsilon]
\end{equation*}
and so,
\begin{equation*}
	\exp(\eta\,\varepsilon t)\{\mathcal{E}^{\bm{h}}[w,v]-\Gamma\Psi(\bm h)\}(t)\leq\{\mathcal{E}^{\bm{h}}[w,v]-\Gamma\Psi(\bm h)\}(0)<0,
	\qquad \quad t\in[0,T_\varepsilon].
\end{equation*}
Therefore, the solution $(u,v)$ remains in the channel $\mathcal{Z}_{{}_{\Gamma,\rho}}$ while $\bm h\in\Omega_\rho$ 
 and if $T_\varepsilon<+\infty$ is maximal, then $\bm h(T_\varepsilon)\in\partial\Omega_\rho$, that is
\begin{equation}\label{hfrontiera}
	h_j(T_\varepsilon)-h_{j-1}(T_\varepsilon)=\varepsilon/\rho \qquad \textrm{for some } j.
\end{equation}
For Proposition \ref{prop:E>}, in the channel the solution satisfies \eqref{umenouh} and \eqref{|h'|<exp-intro}. 
In particular, the transition points move with exponentially small velocity. 
This implies that $(u,v)$ remains in the channel for an exponentially long time. 
Indeed, from \eqref{|h'|<exp-intro} it follows that for all $t\in[0,T_\varepsilon]$, one has
\begin{equation}\label{dhmax}
	|h_j(t)-h_j(0)|\leq C\left(\varepsilon/\tau\right)^{1/2}\exp(-A\ell^{\bm h(t)}/\varepsilon)t \qquad \textrm{for any } j=1,\dots,N,
\end{equation} 
where $\ell^{\bm h(t)}$ is the minimum distance between layers at the time $t$.
Combining \eqref{hfrontiera} and \eqref{dhmax},  we obtain 
\begin{equation*}
	\varepsilon/\rho\geq \ell^{\bm h(0)}-2C(\varepsilon/\tau)^{1/2}\exp(-A/\rho)T_\varepsilon.
\end{equation*}
Hence, using \eqref{triangle} we have
\begin{equation*}
	T_\varepsilon\geq C\bigl(\ell^{\bm h(0)}-\varepsilon/\rho\bigr)(\varepsilon/\tau)^{-1/2}\exp(A/\rho)\geq 
	C\bigl(\ell^{\bm h(0)}-\varepsilon/\rho\bigr)(\varepsilon/\tau)^{-1/2}\exp(A\delta/\varepsilon),
\end{equation*}
and the proof is complete.
\end{proof}

\section{Reduced dynamics on the base manifold}\label{layer}
In the previous section, we derived the equation \eqref{h'-compact} for the motion of the transition points and,
by studying the ODE-PDE coupled system \eqref{system-w-v-h}, we obtained the estimate \eqref{|h'|<exp}
for the velocity of the transitions. 
In this section, we derive an ordinary differential equation approximating the equation for $\bm h$ to obtain further
information on the motion of the transition points 
and analyze the differences with the parabolic case \eqref{AllenCahn}. 

\subsection{Derivation of the reduced system}
Since  $w$ is very small, we use the approximation $w=0$ in \eqref{h'} and then
\begin{equation}\label{h'-w=0}
	\sum_{i=1}^N\langle u^{\bm h}_i,k^{\bm h}_j\rangle\bm h'_i=\langle v,k^{\bm h}_j\rangle, \qquad j=1,\dots,N.
\end{equation}
In order to eliminate $v$, let us differentiate and multiply by $\tau$ equation \eqref{h'-w=0}. 
We have
\begin{align*}
	\tau\sum_{i,l=1}^N & \bigl(\langle u^{\bm h}_{il},k^{\bm h}_j\rangle+\langle u^{\bm h}_i,k^{\bm h}_{jl}\rangle\bigr)\bm h'_l\bm h'_i
	+\tau\sum_{i=1}^N\langle u^{\bm h}_i,k^{\bm h}_j\rangle\bm h''_i=\\
	& -\langle\mathcal{L}(u^{\bm h}),k^{\bm h}_j\rangle-\langle g(u^{\bm h},\tau)v,k^{\bm h}_j\rangle+
	\tau\sum_{l=1}^N\langle v,k^{\bm h}_{jl}\rangle\bm h'_l, \qquad j=1,\dots,N.
\end{align*}
Using the approximation $v=\nabla_{\bm h}u^{\bm h}\cdot\bm h$, we obtain
\begin{align*}
	\tau \sum_{i,l=1}^N & \bigl(\langle u^{\bm h}_{il},k^{\bm h}_j\rangle+\langle u^{\bm h}_i,k^{\bm h}_{jl}\rangle\bigr)\bm h'_l\bm h'_i
	+\tau\sum_{i=1}^N\langle u^{\bm h}_i,k^{\bm h}_j\rangle\bm h''_i=\\
	& -\langle\mathcal{L}(u^{\bm h}),k^{\bm h}_j\rangle-\sum_{i=1}^N\langle g(u^{\bm h},\tau)u^{\bm h}_i,k^{\bm h}_j\rangle\bm h'_i
	+\tau\sum_{i,l=1}^N\langle u^{\bm h}_i,k^{\bm h}_{jl}\rangle\bm h'_i\bm h'_l, \qquad j=1,\dots,N.
\end{align*}
Let us denote by $\nabla^2_{\bm h}u^{\bm h}$ the Hessian of $u^{\bm h}$ with respect to $\bm h$ and 
by $q(\bm \xi):=\displaystyle\sum_{i,l=1}^N u^{\bm h}_{il}\bm \xi_l\bm \xi_i$ the quadratic form associated to $\nabla^2_{\bm h}u^{\bm h}$. 
Simplifying, we get
\begin{equation}\label{h-eq}
	\tau\sum_{i=1}^N\langle u^{\bm h}_i,k^{\bm h}_j\rangle\bm h''_i+\sum_{i=1}^N\langle g(u^{\bm h},\tau)u^{\bm h}_i,k^{\bm h}_j\rangle\bm h'_i
	+\tau\langle q(\bm h'),k^{\bm h}_j\rangle = -\langle\mathcal{L}(u^{\bm h}),k^{\bm h}_j\rangle,
\end{equation}
for $ j=1,\dots,N$. 
If $u^{\bm h(t)}(x)$ is a solution of the hyperbolic Allen--Cahn equation \eqref{hyp-al-ca}, $\bm h(t)$ satisfies \eqref{h-eq}. 
Observe that with respect  to the parabolic case, besides the coefficient  $g(u^{\bm h},\tau)$ which is in general different from 1, there are two new terms: the term involving $\bm h''_i$ and the one involving the quadratic
form associated to the Hessian of $u^{\bm h}$. 
By inverting the matrix $D_{ij}(\bm h)=\langle u^{\bm h}_j,k^{\bm h}_i\rangle$,
introduced in Section \ref{motion}, we rewrite \eqref{h-eq} as follows:
\begin{equation}\label{eq:vecth}
	\tau\bm h''+\mathcal{G}(\bm h)\bm h' +\tau\mathcal{Q}(\bm h,\bm h')=\mathcal{P}(\bm h),
\end{equation}	
where 
\begin{equation*}
	\mathcal{G}_{ij}(\bm h):=\sum_{l=1}^ND^{-1}_{il}(\bm h)\langle g(u^{\bm h},\tau)u^{\bm h}_j,k^{\bm h}_l\rangle, \qquad 
	\mathcal{Q}_i(\bm h,\bm h'):=\sum_{j=1}^ND^{-1}_{ij}(\bm h)\langle q(\bm h'),k^{\bm h}_j\rangle,
\end{equation*}
and
\begin{equation*}
	\mathcal{P}_i(\bm h):=-\sum_{j=1}^N D^{-1}_{ij}(\bm h)\langle\mathcal{L}(u^{\bm h}),k^{\bm h}_j\rangle.
\end{equation*}
Now we want to identify the leading terms in \eqref{eq:vecth}, having in mind the estimates for $k_j^{\bm h}$, $u^{\bm h}$ and their derivatives;
namely we shall rewrite $\mathcal{G}$, $\mathcal{Q}$ and $\mathcal{P}$ by neglecting the exponentially small remainders in the asymptotic expansion for $\varepsilon\to 0$.

As proven by Carr and Pego \cite[Corollary 3.6]{Carr-Pego}, 
defining
\begin{equation*}
	D_\infty:=\int_{-1}^1\sqrt{2F(s)}\,ds \quad \textrm{and} \quad 
	\mathcal{P}^*_j(\bm h):=-\varepsilon D_\infty^{-1}\langle\mathcal{L}(u^{\bm h}),k^{\bm h}_j\rangle=\varepsilon D_\infty^{-1}(\alpha^{j+1/2}-\alpha^{j-1/2}),
\end{equation*}
there exists $C>0$ such that if $\rho$ is sufficiently small and $\bm h\in\Omega_\rho$, we have
\begin{equation}\label{eq:P-P*}
	|\mathcal{P}(\bm h)-\mathcal{P}^*(\bm h)|_{{}_{\infty}}\leq C|\mathcal{P}^*(\bm h)|_{{}_{\infty}}\exp(-A\ell^{\bm h}/2\varepsilon),
\end{equation}
where $|\mathcal{P}(\bm h)|_{{}_{\infty}}=\max|\mathcal{P}_j(\bm h)|$. 
Since $\alpha^{j-1/2}=F(\phi^j)-\frac12\varepsilon^2\bigl(\phi^j_x\bigr)^2=F(\phi^j(h_{j-1/2}))$, 
it follows that $\mathcal{P}_j(\bm h)$ depends essentially on the differences between values of the potential $F(\phi)$. 

Similar result holds for   the matrix $\mathcal{G}(\bm h)$, namely for the scalar products  
$\langle g(u^{\bm h},\tau)u^{\bm h}_i,k^{\bm h}_j\rangle$, thus generalizing the aforementioned result to the case $g\not\equiv 1$.
To this end, we recall the following result  (see \cite[Lemmas 7.8-7.9-8.1]{Carr-Pego}).

\begin{lem}[Carr--Pego \cite{Carr-Pego}]\label{lem:u^h_j}
The interval $[h_{j-1}-\varepsilon,h_{j+1}+\varepsilon]$ contains the support of $u^{\bm h}_j$ and
\begin{equation*}
	u^{\bm h}_j=
		\begin{cases}
			\chi^{j-1}\nu^j \qquad \quad & x\in I_{j-1},\\
			(1-\chi^j)(-\phi^j_x+\nu^j)+\chi^j(-\phi^{j+1}_x-\nu^{j+1})\\
			+\chi^j_x(\phi^j-\phi^{j+1}) & x\in I_j,\\
			-(1-\chi^{j+1})\nu^{j+1} & x \in I_{j+1},
		\end{cases}
\end{equation*}
where $\nu^j(x):=\nu(x-h_{j-1/2},h_j-h_{j-1}, (-1)^j)$ for $x\in I_j$ and there exists $r_0>0$ such that, for $0<r<r_0$,  
\begin{equation}\label{eq:nu}
	|\nu(x,\ell,\pm1)|\leq C\varepsilon^{-1}\beta_\pm(r), \qquad \textrm{for } x\in\left[-\tfrac\ell2-\varepsilon,\tfrac\ell2+\varepsilon\right].
\end{equation}
\end{lem}

In order to compute $u^{\bm h}_j = \partial_{h_j} u^{\bm h}$, one needs to obtain an expression for  $\phi_\ell$.
Since $\phi$ solution of \eqref{fi(x,l)} depends on $\ell$ through its boundary value, the latter can be obtained by   differentiating the integrated version of that equation with respect to $\ell$, that is, the $\varepsilon$--rescaled version of \eqref{eq:varfi}. 
Finally, for $x\in[-\ell,\ell]$, we end up with 
\begin{equation*}
	 \phi_\ell(x,\ell,\pm1)=  \nu(x,\ell,\pm1) - \tfrac12(\textrm{sgn } x)\phi_x(x,\ell,\pm1),
\end{equation*}
and $\nu$ is an even function of $x$ satisfying
\begin{equation*}
	\varepsilon^2\nu_{xx}=f'(\phi)\nu, \qquad \quad \textrm{for } \, x\in[0,\ell];
\end{equation*}
 see \cite[Lemma 7.8]{Carr-Pego} for details.
 
From \eqref{eq:nu}, it follows that 
\begin{equation*}
	|\nu^j(x)|\leq C\varepsilon^{-1}\beta^{j-1/2}, \qquad \textrm{for }x\in[h_{j-1}-\varepsilon,h_j+\varepsilon],
\end{equation*}
and so, for $x\in[h_{j-1}-\varepsilon,h_{j+1}+\varepsilon]$,
\begin{equation}\label{eq:nuj}
	|(1-\chi^j)\nu^j|+|\chi^j\nu^{j+1}|\leq C\varepsilon^{-1}\max\{\beta^{j-1/2},\beta^{j+1/2}\} \leq C\varepsilon^{-1}\exp(-A\ell^{\bm h}/2\varepsilon).
\end{equation}

Note that, for $x\in I_j$, one has
\begin{equation}\label{eq:perujj}
u^{\bm h}_x=(1-\chi^j)\phi^j_x+\chi^j\phi^{j+1}_x+\chi^j_x(\phi^{j+1}-\phi^j)\ \hbox{and}\
	u^{\bm h}_j=-u^{\bm h}_x+(1-\chi^j)\nu^j-\chi^j\nu^{j+1}.
\end{equation}

Thanks to Lemma \ref{lem:u^h_j}, we can prove the following proposition.
\vspace{0.2cm}
\begin{prop}\label{prop:scal-prod}
Let $F\in C^3$ be such that \eqref{hypf} holds and $g\in C^1(\mathbb R)$. Set
\begin{equation*}	
	 C_{F,g}:=\int_{-1}^1\sqrt{2F(s)}g(s)ds.
\end{equation*}
If $\rho$ is sufficiently small and $\bm h\in\Omega_\rho$, then there exists $C>0$ such that, for $j=1,\dots,N$,
\begin{align}
	&\bigl|\langle g(u^{\bm h})u_j^{\bm h},k^{\bm h}_j\rangle-\varepsilon^{-1}C_{F,g}\bigr|
		\leq C\varepsilon^{-1}\max\{\beta^{j-1/2},\beta^{j+1/2}\}
		\leq C\varepsilon^{-1}\exp(-A\ell^{\bm h}/2\varepsilon), \label{g(u^h)u_j,k_j}\\
	&\bigl|\langle g(u^{\bm h})u_j^{\bm h},k^{\bm h}_{j+1}\rangle\bigr|+\bigl|\langle g(u^{\bm h})u_{j+1}^{\bm h},k^{\bm h}_j\rangle\bigr|
		\leq C\varepsilon^{-1}\beta^{j+1/2}
		\leq C\varepsilon^{-1}\exp(-A\ell^{\bm h}/2\varepsilon),\label{g(u^h)u_j-j+1}\\ 
	&\langle g(u^{\bm h})u_j^{\bm h},k^{\bm h}_i\rangle=0 \qquad\mbox{ if } |j-i|>1. \label{g(u^h)u_j,k_i}
\end{align}
\end{prop}

\begin{proof}
Firstly, we recall that $k^{\bm h}_j$ is supported in $I_j$.
Since the support of $u^{\bm h}_j$ is contained in $[h_{j-1}-\varepsilon,h_{j+1}+\varepsilon]$, we have \eqref{g(u^h)u_j,k_i}. 
From Lemma \ref{lem:u^h_j}, it follows that $|u^{\bm h}_{j+1}(x)|\leq C\varepsilon^{-1}\beta^{j+1/2}$, for $x\in I_j$ 
and $|u^{\bm h}_j(x)|\leq C\varepsilon^{-1}\beta^{j+1/2}$ for $x\in I_{j+1}$. 
Then,
\begin{equation*}
	\bigl|\langle g(u^{\bm h})u_j^{\bm h},k^{\bm h}_{j+1}\rangle\bigr|+\bigl|\langle g(u^{\bm h})u_{j+1}^{\bm h},k^{\bm h}_j\rangle\bigr|\leq 
	C\varepsilon^{-1}\beta^{j+1/2}\left(\int_{I_j}|k^{\bm h}_{j}|+\int_{I_{j+1}}|k_{j+1}^{\bm h}|\right).
\end{equation*}
However, $u^{\bm h}_x$ is of one sign in $I_j$, thus $\int_{I_j}|k^{\bm h}_{j}|\leq C$ and we obtain \eqref{g(u^h)u_j-j+1}.
It remains to prove \eqref{g(u^h)u_j,k_j}. 
To do this, for $x\in I_j$, we write $u_j^{\bm h}=y_1+y_2$ and $k^{\bm h}_j=y_1+y_3$, 
where 
\begin{align*}
	y_1&=-(1-\chi^j)\phi^j_x-\chi^j\phi^{j+1}_x,\\
	y_2&=-\chi^j_x(\phi^{j+1}-\phi^j)+(1-\chi^j)\nu^j-\chi^j\nu^{j+1}, \\
	y_3&=(1-\gamma^j)u^{\bm h}_x-\chi^j_x(\phi^{j+1}-\phi^j). 
\end{align*}
Then,
\begin{equation*}
	\langle g(u^{\bm h})u_j^{\bm h},k^{\bm h}_j\rangle=\int_{h_{j-1/2}}^{h_{j+1/2}}g(u^{\bm h}(x))\bigl\{y_1(x)^2+y_1(x)(y_2(x)+y_3(x))+y_2(x)y_3(x)\bigr\}dx.
\end{equation*}
From \eqref{u^h(x)} and the definition of $\chi$, it follows that
\begin{equation}\label{eq:cor1}
	\int_{h_{j-1/2}}^{h_{j+1/2}}g(u^{\bm h})y_1^2 dx=\int_{h_{j-1/2}}^{h_j}g(\phi^j)\left(\phi^j_x\right)^2dx+\int_{h_j}^{h_{j+1/2}}g(\phi^{j+1})\left(\phi^{j+1}_x\right)^2dx+E,
\end{equation}
where 
\begin{equation*}
	E=\int_{h_j-\varepsilon}^{h_j+\varepsilon}g(u^{\bm h})y_1^2dx-\int_{h_j-\varepsilon}^{h_j}g(\phi^j)(\phi^j_x)^2dx-\int_{h_j}^{h_j+\varepsilon}g(\phi^{j+1})(\phi_x^{j+1})^2dx.
\end{equation*}
By writing $y_1=\chi^j(\phi^j_x-\phi^{j+1}_x)-\phi^j_x=(1-\chi^j)(\phi^{j+1}_x-\phi^j_x)-\phi^{j+1}_x$, we get
\begin{equation*}
	E=\int_{h_j-\varepsilon}^{h_j}\{g(u^{\bm h})-g(\phi^j)\}(\phi^j_x)^2dx+\int_{h_j}^{h_j+\varepsilon}\{g(u^{\bm h})-g(\phi^{j+1})\}(\phi^{j+1}_x)^2dx+R.
\end{equation*}
Using \eqref{fi^j-fi^j+1} and the estimate $|\phi^j_x|\leq C\varepsilon^{-1}$, we deduce that $R$ satisfies
\begin{align*}
	|R|\leq C\biggl(\int_{h_j-\varepsilon}^{h_j}g(u^{\bm h})|\phi^j_x||\phi^{j+1}_x-\phi^j_x|+\int_{h_j}^{h_j+\varepsilon}g(u^{\bm h})|\phi^{j+1}_x||\phi^{j+1}_x-\phi^j_x|\\
	+\int_{h_j-\varepsilon}^{h_j+\varepsilon}g(u^{\bm h})|\phi^{j+1}_x-\phi^j_x|^2dx\biggr)\leq C\varepsilon^{-1}\max\{\alpha^{j-1/2},\alpha^{j+1/2}\}.
\end{align*}
Moreover, for \eqref{fi^j-fi^j+1} we have
\begin{align*}
	\biggl|\int_{h_j-\varepsilon}^{h_j}\{g(u^{\bm h})-g(\phi^j)\}(\phi^j_x)^2dx\biggr| & \leq C\int_{h_j-\varepsilon}^{h_j}|u^{\bm h}-\phi^j|(\phi^j_x)^2 dx \\
	& \leq C\varepsilon^{-2}\int_{h_j-\varepsilon}^{h_j}\chi^j|\phi^{j+1}-\phi^j|\leq C\varepsilon^{-1}|\alpha^{j-1/2}-\alpha^{j+1/2}|.
\end{align*}
Similarly, we can estimate the other term and obtain $|E|\leq C\varepsilon^{-1}\max\{\alpha^{j-1/2},\alpha^{j+1/2}\}$. 

Let us now compute 
\begin{equation*}
	\int_{h_{j-1/2}}^{h_j}g(\phi^j)(\phi^j_x)^2dx\,;
\end{equation*}
 the other remaining term in \eqref{eq:cor1} is evaluated similarly.
To do this, we observe that since $\phi^j(x)=\phi(x-h_{j-1/2},h_{j} - h_{j-1},(-1)^j)$
and $\phi(x,\ell,\pm1)$ is solution of \eqref{fi(x,l)}, positive or negative respectively, we have
\begin{equation}\label{fi_x^2-F}
	\varepsilon^2(\phi^j_x)^2=2(F(\phi^j)-\alpha^{j-1/2}),
\end{equation}
where $\alpha^{j-1/2}=F(\phi^j(h_{j-1/2 }))$. 
In what follows, we are considering the case $\phi^j(x)<0$ in $[h_{j-1/2},h_{j}]$ (i.e. $j$ odd); 
treatment of $\phi(x-h_{j-1/2},h_{j} - h_{j-1},+1)$ is similar. 
By using \eqref{fi_x^2-F} and changing variable, we obtain 
\begin{align*}
	\varepsilon\int_{h_{j-1/2}}^{h_j}g(\phi^j)(\phi^j_x)^2dx & =\int_{\phi^j(h_{j-1/2})}^{0}g(s)\sqrt{2(F(s)-\alpha^{j-1/2})}\,ds\\
	& \qquad -\int_{\phi^j(h_{j-1/2})}^{0}g(s)\sqrt{2F(s)}\,ds+\int_{\phi^j(h_{j-1/2})}^{0}g(s)\sqrt{2F(s)}\,ds\\
	&= \int_{-1}^{0}g(s)\sqrt{2F(s)}\,ds-\int_{-1}^{\phi^j(h_{j-1/2})}g(s)\sqrt{2F(s)}\,ds\\
	& \qquad -\sqrt2\int_{\phi^j(h_{j-1/2})}^{0}\frac{\alpha^{j-1/2} g(s)}{\sqrt{F(s)-\alpha^{j-1/2}}+\sqrt{F(s)}}\,ds.
\end{align*}
Since $F(s)\geq\alpha^{j-1/2}$ for $\phi^j(h_{j-1/2})\leq s\leq0$ and $F(s)\leq\alpha^{j-1/2}$ for $-1\leq s\leq\phi^j(h_{j-1/2})$, we have 
\begin{align*}
	\int_{-1}^{\phi^j(h_{j-1/2})}g(s)\sqrt{2F(s)}\,ds+\sqrt2\int_{\phi^j(h_{j-1/2})}^{0}\frac{\alpha^{j-1/2} g(s)}{\sqrt{F(s)-\alpha^{j-1/2}}+\sqrt{F(s)}}\,ds\\
	\leq C\sqrt{\alpha^{j-1/2}} \leq C \beta^{j-1/2}.
\end{align*}
Then, we can conclude that
\begin{equation*}
	\left | \int_{h_{j-1/2}}^{h_{j+1/2}}g(u^{\bm h})y_1^2 dx- \varepsilon^{-1}C_{F,g} \right | \leq C\varepsilon^{-1}\max\left\{\beta^{j-1/2},\beta^{j+1/2}\right\}.
\end{equation*}
Moreover, from \eqref{fi^j-fi^j+1} and \eqref{eq:nuj}, it follows that for $x\in[h_{j-1/2},h_{j+1/2}]$ we have
\begin{equation*}
	|y_2|\leq C\varepsilon^{-1}\max\{\beta^{j-1/2},\beta^{j+1/2}\}, \qquad |y_3|\leq|(1-\gamma^j)u^{\bm h}_x|+C\varepsilon^{-1}|\alpha^{j-1/2}-\alpha^{j+1/2}|.
\end{equation*}
We claim that
\begin{equation}\label{eq:y3}
	|(1-\gamma^j)u^{\bm h}_x|\leq C\varepsilon^{-1}\max\{\beta^{j-1/2},\beta^{j+1/2}\}.
\end{equation}
Indeed, $(1-\gamma^j)u^{\bm h}_x=0$ in $[h_{j-1/2}+2\varepsilon, h_{j+1/2}-2\varepsilon]$ and   
\begin{equation*}
	u^{\bm h}_x(x)=\left\{
	\begin{aligned}	
	&\phi^j_x(x) \qquad \qquad & x\in[h_{j-1/2},h_{j-1/2}+2\varepsilon],\\
	&\phi^{j+1}_x(x) \qquad \qquad & x\in[h_{j+1/2}-2\varepsilon,h_{j+1/2}].
	\end{aligned}\right.
\end{equation*}
Using the fact that
\begin{equation*}
	 |\phi_x(x,\ell,\pm1)|\leq C\varepsilon^{-1}\sqrt{F(\phi(x,\ell,\pm1))}\leq C\varepsilon^{-1}(1\mp\phi(0,\ell,\pm1))=C\varepsilon^{-1}\beta_\pm(r), 
\end{equation*}
for $|x|\leq2\varepsilon$, we obtain \eqref{eq:y3} and so, $|y_3|\leq C\varepsilon^{-1}\max\{\beta^{j-1/2},\beta^{j+1/2}\}$. 
Therefore,
\begin{align*}
	\biggl|\int_{h_{j-1/2}}^{h_{j+1/2}}g(u^{\bm h})y_1(y_2+y_3)\biggr|&\leq C\varepsilon^{-1}\max\{\beta^{j-1/2},\beta^{j+1/2}\}\int_{h_{j-1/2}}^{h_{j+1/2}}|y_1|\\
	&\leq C\varepsilon^{-1}\max\{\beta^{j-1/2},\beta^{j+1/2}\}.
\end{align*}
Also,
\begin{equation*}
	\biggl|\int_{h_{j-1/2}}^{h_{j+1/2}}g(u^{\bm h})y_2y_3\biggr|\leq C\varepsilon^{-2}\max\{\beta^{j-1/2},\beta^{j+1/2}\}\varepsilon
	\leq C\varepsilon^{-1}\max\{\beta^{j-1/2},\beta^{j+1/2}\},
\end{equation*}
because $y_3$ is supported on a set of measure proportional to $\varepsilon$.
\end{proof}

We are ready to analyze the term  $\mathcal{G}(\bm h)$. To this aim, let us introduce the constant  
\begin{equation*}
	\gamma_\tau:=\frac{\sqrt{2}}{D_\infty}\int_{-1}^1\sqrt{F(s)}g(s,\tau)\,ds = \frac{C_{F,g}}{D_\infty}.
\end{equation*} 
Then, in view of \eqref{g(u^h)u_j,k_j}--\eqref{g(u^h)u_j-j+1}--\eqref{g(u^h)u_j,k_i}, we obtain
\begin{equation*}
	|\langle g(u^{\bm h},\tau)u_i^{\bm h},k^{\bm h}_j\rangle| \leq C\varepsilon^{-1}, \ \hbox{for any}\ i,j,
\end{equation*}
and, being
\begin{equation*}
D^{-1}(\bm h) = \varepsilon D_\infty^{-1} \Big \{\mathbb{I}_N -  \big (  \mathbb{I}_N - \varepsilon D^{-1} _\infty D(\bm h)\big ) \Big\}^{-1}
	= \varepsilon D_\infty^{-1}\sum_{k=0}^\infty\Big\{\varepsilon D_\infty^{-1}\big(\varepsilon ^{-1}D_\infty\mathbb{I}_N-D(\bm h)\big)\Big\}^k,
\end{equation*}
one has
\begin{equation*}
	\|D^{-1}(\bm h)-\varepsilon D_\infty^{-1} \mathbb I_N \|_{{}_{\infty}}\leq\varepsilon C\exp(-A\ell^{\bm h}/2\varepsilon).
\end{equation*}
Hence, for $g=g(u^{\bm h},\tau)$,
\begin{align}
	|\mathcal{G}_{ij}(\bm h)-\gamma_\tau\mathbb \delta_{ij} |& \leq \left | 
	\sum_{l=1}^N (D^{-1}_{il}(\bm h) - \varepsilon  D^{-1}_\infty \delta_{il} ) \langle g\,u^{\bm h}_j,k^{\bm h}_l\rangle \right |
	+ \left | \varepsilon D^{-1}_\infty  \langle g\,u^{\bm h}_j,k^{\bm h}_i\rangle -  \frac{C_{F,g}}{D_\infty} \delta_{ij} \right | \nonumber \\ 
	& \leq C\exp(-A\ell^{\bm h}/2\varepsilon). \label{eq:G-G*}
\end{align}
Therefore, in \eqref{eq:vecth} we substitute the matrix $\mathcal{G}(\bm h)$ with $\gamma_\tau\mathbb I_N$.

Let us now focus our attention on the term $\tau\mathcal{Q}(\bm h,\bm h')$; analogously to the previous terms we have
\begin{equation*}
	|\mathcal{Q}(\bm h,\bm h')-\mathcal{Q}^*(\bm h,\bm h')|_{{}_{\infty}}\leq C |\mathcal{Q}^*(\bm h,\bm h')|_{{}_{\infty}}\exp(-A\ell^{\bm h}/2\varepsilon),
\end{equation*}
where 
\begin{equation*}
	\mathcal{Q}_j^*(\bm h,\bm h'):=\varepsilon D_\infty^{-1}\sum_{i,l=1}^N \langle u^{\bm h}_{il},k^{\bm h}_j\rangle\bm h'_l\bm h'_i.
\end{equation*}
Then, let us study the elements $\langle u^{\bm h}_{il},k^{\bm h}_j\rangle$, 
which shall be treated in a similar way of Proposition \ref{prop:scal-prod}. 
Since $k^{\bm h}_j$ is supported in $I_j$, it follows that $\langle u^{\bm h}_{il},k^{\bm h}_j\rangle=0$ if either $|i-j|>1$ or $|l-j|>1$.
For all the remaining terms, from the expression of $u^{\bm h}_j$ in Lemma \ref{lem:u^h_j}, 
and using the bounds in \cite{Carr-Pego, Carr-Pego2}, the only one which may not be exponentially small is for $i=l=j$. 
Therefore, here we omit the tedious, but straightforward control of such terms and we discuss only $\langle u^{\bm h}_{jj},k^{\bm h}_j\rangle$.
To this end, observe that in the interval $I_j$, by differentiating \eqref{eq:perujj} with respect to $h_j$, we have
\begin{align*}
u^{\bm h}_{xj}&=\phi^j_{jx}+\chi^j(\phi^{j+1}_{jx}-\phi^j_{jx})+\chi^j_{xx}(\phi^j-\phi^{j+1})+
	\chi^j_x[(\phi^{j+1}_j-\phi^j_j)-(\phi_x^{j+1}-\phi_x^j)],\\
	u^{\bm h}_{jj}&=-u^{\bm h}_{xj}+\nu_j^j-\chi^j(\nu^j_j+\nu_j^{j+1})-\chi_j^j(\nu^j+\nu^{j+1}).
\end{align*}
Using 
\begin{equation*}
	\phi_j^j=-\phi^j_x+\nu^j, \quad \phi_j^{j+1}=-\phi_x^{j+1}-\nu^{j+1}, \quad \phi_{jx}^j=-\phi^j_{xx}+\nu^j_x,
	\quad \phi_{jx}^{j+1}=-\phi_{xx}^{j+1}-\nu^{j+1}_x,
\end{equation*}
 and from the expression of $u^{\bm h}_{xx}$ obtained again from \eqref{eq:perujj},  in $x\in I_j$ we infer
\begin{equation*}
	u^{\bm h}_{xj}=-u^{\bm h}_{xx}+\nu^j_x-\chi^j(\nu_x^{j+1}+\nu^j_x)-\chi^j_x(\nu^{j+1}+\nu^j).
\end{equation*}
Hence, we can conclude that $u^{\bm h}_{jj}=u^{\bm h}_{xx}+R^j$,
and so
\begin{equation*}
	\langle u^{\bm h}_{jj},k^{\bm h}_j\rangle=-\int_{I_j}u^{\bm h}_{xx}u^{\bm h}_x\,dx+
	\int_{I_j}u^{\bm h}_{xx}(1-\gamma^j)u^{\bm h}_x\,dx+\int_{I_j}R^jk^{\bm h}_j\,dx.
\end{equation*}
Reasoning as in the proof of the Proposition \ref{prop:scal-prod}, and taking into account the needed bounds for the higher involved derivatives \cite{Carr-Pego, Carr-Pego2}, 
one can prove that the last two integrals are exponentially small, whereas for the first integral we obtain
\begin{equation*}
	\int_{h_{j-1/2}}^{h_{j+1/2}}\bigl(-u^{\bm h}_{xx}(x)u^{\bm h}_x(x)\bigr)dx=
	\frac12\bigl(u^{\bm h}_x(h_{j-1/2})^2-u^{\bm h}_x(h_{j+1/2})^2\bigr)=0,
\end{equation*}
because $u^{\bm h}_x(h_{j-1/2})=0$ for all $j=1,\dots,N$.
Therefore, we obtain that there exists $c>0$ such that
\begin{equation}\label{eq:Q*}
	|\mathcal{Q}(\bm h,\bm h')|_{{}_{\infty}}\leq C\exp(-c/\varepsilon)|\bm h'|^2_{{}_{\infty}}.
\end{equation}
In conclusion, using the estimates \eqref{eq:P-P*}, \eqref{eq:G-G*} and \eqref{eq:Q*}, and neglecting all the exponentially small terms in \eqref{eq:vecth}, 
we end up with the reduced system 
\begin{equation}\label{h-eq-approx}
	\tau\bm h''+\gamma_\tau\bm h'=\mathcal{P}^*(\bm h).
\end{equation}
In the case of the damped wave equation with bistable nonlinearity, $g(u,\tau)\equiv 1$, and therefore $\gamma_\tau=1$. 
Moreover, for the Allen--Cahn equation with relaxation, $g(u,\tau) = 1 + \tau f'(u)$ and 
 \begin{align*}
	\gamma_\tau & 
		= 1 + \frac{\sqrt{2}\,\tau}{D_\infty} \int_{-1}^1\sqrt{F(s)}F''(s)\,ds 
		= 1 - \frac{\tau}{D_\infty}\int_{-1}^1\displaystyle{\frac{(F'(s))^2}{\sqrt{2F(s)}}}\,ds 
	< 1.
\end{align*} 
Hence, in the latter case, the effect of the parameter $\tau>0$ is present also in the friction term  $\gamma_\tau\bm h'$,
and in particular it speeds up the dynamics with respect to the simpler nonlinear damped wave equation, being the coefficient smaller. 
This richer effect on the dynamics 
in the present analysis confirms what has been already observed in the study of traveling waves  in 
 \cite{LMPS}, where again the relaxation parameter $\tau$ in the case of   the Allen--Cahn equation with relaxation
  affects the speed of the wave also though a modification of the friction effects. 

\subsection{Comparison with the parabolic case}
Now, if $\gamma_\tau\to1$ as $\tau\rightarrow0$, taking formally  the limit in \eqref{h-eq-approx} we obtain the
system $\bm h'=\mathcal{P}^*(\bm h)$. 
The structure of solutions of this system of ordinary differential equations is studied in \cite{Carr-Pego}
to describe the evolution of the layer positions in the parabolic case \eqref{AllenCahn}. 
We can write
\begin{equation*}
	\mathcal{P}^*(\bm h)=-\nabla W(\bm h),
\end{equation*}
where $W$ is defined in the following way. 
For $s>\rho^{-1}$, define $W_\pm$ by $W'_\pm(s)=D^{-1}_\infty\alpha_\pm(s^{-1})$
and set $W_j=W_+$ for $j$ even, $W_j=W_-$ for $j$ odd. 
For $\bm h\in\Omega_\rho$ let
\begin{equation*}
	W(\bm h):=\varepsilon^2\biggl[\tfrac12W_1((h_1-h_0)/\varepsilon) +\sum_{j=2}^NW_j((h_j-h_{j-1})/\varepsilon) 
		+\tfrac12W_{N+1}((h_{N+1}-h_N)/\varepsilon)\biggr].
\end{equation*}
Since $h_0=-h_1$ and $h_{N+1}=2-h_N$, we have
\begin{equation*}
\begin{aligned}
	\frac{\partial W(\bm h)}{\partial h_j} & = \varepsilon\bigl[W'_j((h_j-h_{j-1})/\varepsilon)-W'_{j+1}((h_{j+1}-h_j)/\varepsilon)\bigr]\\
	& = -\varepsilon D_\infty^{-1}(\alpha^{j+1/2}-\alpha^{j-1/2}).
\end{aligned}
\end{equation*}
Then, we can write \eqref{h-eq-approx} as $\tau\bm h''+\bm h'=-\nabla W(\bm h)$. 
For a solution $\bm h$ with values in $\Omega_\rho$, the energy $\mathsf{E}_\tau=\frac12\tau|\bm h'|^2+W(\bm h)$ is nonincreasing and
\begin{equation*}
	\frac{d\mathsf{E}_\tau}{dt}=\tau\bm h'\cdot\bm h''+\nabla W(\bm h)\cdot\bm h'=-\gamma_\tau|\bm h'|^2\leq 0.
\end{equation*}
Note that $\gamma_\tau$ large implies a greater dissipation of energy.
\vspace{0.2cm}
\begin{prop}[Carr--Pego \cite{Carr-Pego}] \label{lem-W}
If $\rho$ is sufficiently small, then the function $W$ has a unique critical point $\bm h^e$, which is a strict local maximum.
\end{prop}

\begin{proof}
If $\bm h^e$ is a critical point of $W$, then $\alpha^{j-1/2}=\alpha^{j+1/2}$ for $j=1,\dots,N$. 
Define $\ell_j:=h_j-h_{j-1}$. 
From Proposition \ref{prop:alfa,beta}, it follows that for $r$ sufficiently small $\alpha_\pm(r)$ are monotone and so, 
$\ell_j=\ell_{j+2}$ for $j=1,\dots,N-1$. Let $\ell_-=\ell_1$, $\ell_+=\ell_2$. 
Since $\bm h\in\Omega_\rho$, we have that $\frac12\ell_1+\ell_2+\dots+\ell_N+\frac12\ell_{+1}=1$ 
and then $\ell_-+\ell_+=2/N$. 
The condition $\alpha^{j-1/2}=\alpha^{j+1/2}$, also, gives $\alpha_+(\varepsilon/\ell_+)=\alpha_-(\varepsilon/\ell_-)$. 
Using Proposition \ref{prop:alfa,beta}, we have that critical points correspond to zeros of 
\begin{equation*}
	\gamma=\varepsilon^{-1}(A_-\ell_--A_+\ell_+)+2\ln(K_+A_+/K_-A_-)+O\bigl(\rho^{-1} \exp(-A/2\rho)\bigr),
\end{equation*}
for $\ell_-\in[\varepsilon\rho^{-1}, 2N^{-1}-\varepsilon\rho^{-1}]$ with $\ell_+=2N^{-1}-\ell_-$. 
For $\rho$ sufficiently small, $\gamma>0$ when $\ell_-=\varepsilon\rho^{-1}$ and $\gamma<0$ when $\ell_-=2N^{-1}-\varepsilon\rho^{-1}$; 
hence a critical point must exist. 
It is unique, because $\alpha_+$ and $\alpha_-$ are monotone. 

To determine the nature of the critical point, consider the Hessian of $W$, $B_{ij}=\partial^2W(\bm h^e)/\partial h_i\partial h_j$. 
The matrix $B$ is symmetric and tri-diagonal with 
\begin{equation*}
\begin{aligned}
B_{11} & =2\omega_1+\omega_2, & \qquad B_{jj}&=\omega_j+\omega_{j+1}, \quad  &\textrm{for}\, j=2,\dots,N-1,\\
B_{NN} & =\omega_N+2\omega_{N+1}, & \qquad  B_{j,j+1}&=-\omega_{j+1}, \quad  &\textrm{for}\, j=1,\dots,N-1,
\end{aligned}
\end{equation*}
where $\omega_j=W''_j((h_j^e-h^e_{j-1})/\varepsilon)$. 
From Proposition \ref{prop:alfa,beta}, $w_j<0$ for $j=1,\dots,N$. Then, for $y\in\mathbb{R}^N$,
\begin{equation*}
	\sum_{i,j=1}^N B_{ij}y_iy_j=2\omega_1y_1^2+\sum_{j=2}^N\omega_j(y_{j-1}-y_j)^2+2\omega_{N+1}y_N^2,
\end{equation*}
so that $B$ is negative definite and $\bm h^e$ is a local maximum.
\end{proof}

If $f$ is odd, all the $\pm$ subscripts can be ignored, e.g. $\alpha_+=\alpha_-$, $\ell_+=\ell_-$ and
$h_j^e-h^e_{j-1}=1/N$ for $j=1,\dots,N$. 
In general, the steady state domain lengths satisfy
\begin{equation*}
\begin{aligned}
	\ell_- & =2\left(A_+/N+\varepsilon\ln(K_-A_-/K_+A_+)\right)/(A_++A_-)+O\left(\rho^{-1}\exp(-A/2\rho)\right),\\
	\ell_+ & =2\left(A_-/N+\varepsilon\ln(K_+A_+/K_-A_-)\right)/(A_++A_-)+O\left(\rho^{-1}\exp(-A/2\rho)\right).
\end{aligned}
\end{equation*}
Thus, system \eqref{h-eq-approx} has a unique equilibrium point $(\bm h^e,0)$ with $\bm h^e\in\Omega_\rho$. 
From Proposition \ref{prop-L(u^h)}, $\mathcal{L}(u^{\bm h^e})=0$ and so $u^{\bm h^e}$ is a stationary solution of \eqref{hyp-al-ca}.
In other words, there is a unique stationary solution $u^e$ of \eqref{hyp-al-ca} with $N$ transition layers and $u(0)<0$; 
moreover, $u^e\in\mathcal{M}$ with $u^e=u^{\bm h^e}$ where $\bm h^e=\mathcal{H}(u^e)$. 
Note, also, that by definition \eqref{Psi(h)}, $\Psi(\bm h)=0$ if and only if $\bm h=\bm h^e$ so that
the channel $\mathcal{Z}_{{}_{\Gamma,\rho}}$ is ``pinched'' at $u=u^e$.
\vskip.25cm

Now, let us study the stability of the equilibrium point $(\bm h^e,0)$ for system \eqref{h-eq-approx}.
To do this, rewrite it as the first order system
\begin{equation}\label{h-system}
\begin{cases}
	\bm h'=\bm\eta, \\
	\tau\bm\eta'=\mathcal{P}^*(\bm h)-\gamma_\tau\bm\eta.
\end{cases}
\end{equation}

\begin{prop}
System \eqref{h-system} has a unique equilibrium point $(\bm h^e,0)$, which is unstable. 
In particular, the Jacobian matrix evaluated at $(\bm h^e,0)$ has $N$ negative eigenvalues and $N$ positive eigenvalues.
\end{prop}

\begin{proof}
From Lemma \ref{lem-W}, it follows that the system \eqref{h-system} has a unique equilibrium point $(\bm h^e,0)$.
To determine the stability of this stationary point, we have to analyze the eigenvalues of the block matrix
\begin{equation*}
	J=\left(\begin{matrix} 0_N & \mathbb{I}_N \\
	-\frac1\tau B & -\frac{\gamma_\tau}\tau\mathbb{I}_N \end{matrix} \right),
\end{equation*}
where, as above, $B$ is the Hessian matrix of $W$ evaluated at $\bm h^e$. 
To this end, we make use of the Schur complement, defined for a general block   $n\times m$ matrix 
\begin{equation*}
M = 
\begin{pmatrix}
A _1 & A_2 \\
A_3  & A_4 \\
\end{pmatrix}
\end{equation*}
as follows: $ M/A_1 = A_4 - A_3A_1^{-1}A_2$, provided $A_1$ is an invertible square matrix.
In this case, $\det(M) = \det(A_1)\det(M/A_1)$. 
For $M_\lambda=J-\lambda\mathbb I_{2N}$, we have  
\begin{equation*}
	\det(M_\lambda)=\det\left(\begin{matrix} -\lambda\mathbb I_N & \mathbb{I}_N \\
	-\frac1\tau B & -\left(\frac{\gamma_\tau}\tau+\lambda\right)\mathbb{I}_N \end{matrix} \right)
	=\det(-\lambda\mathbb I_N)\det(M_\lambda/(-\lambda\mathbb I_N)),
\end{equation*}
where 
\begin{equation*}
	M_\lambda/(-\lambda\mathbb I_N)=-\left(\frac{\gamma_\tau}\tau+\lambda\right)\mathbb{I}_N-\frac{1}{\tau\lambda} B.
\end{equation*}
Then,
\begin{equation*}
	\det M_\lambda=(-\lambda)^N\det\left(-\frac{1}{\tau\lambda}B-\left(\frac{\gamma_\tau+\tau\lambda}\tau\right)\mathbb{I}_N\right)=
	\frac1{\tau^N}\det\left(B+(\gamma_\tau\lambda+\tau\lambda^2)\mathbb I_N\right).
\end{equation*}
It follows that $\lambda$ is an eigenvalue of $J$ if and only if $-\tau\lambda^2-\gamma_\tau\lambda$ is an eigenvalue of $B$. 
As previously shown in Lemma \ref{lem-W}, $B$ is symmetric and negative definite, so all the eigenvalues of $B$ are negative. 
Denote them by $-\mu_i^2$ for $i=1,\dots,N$. 
For each $-\mu_i^2$ there are two eigenvalues of $J$:
\begin{equation*}
	\lambda_i^+(\tau)=\frac{-\gamma_\tau+\sqrt{\gamma_\tau^2+4\tau\mu_i^2}}{2\tau}>0, \qquad \quad 
	\lambda_i^-(\tau)=\frac{-\gamma_\tau-\sqrt{\gamma_\tau^2+4\tau\mu_i^2}}{2\tau}<0.
\end{equation*}
In conclusion, the Jacobian matrix evaluated at $(\bm h^e,0)$ has $N$ positive eigenvalues $\lambda_i^+(\tau)$ 
and $N$ negative eigenvalues $\lambda_i^-(\tau)$, and so $(\bm h^e,0)$ is unstable. 
The eigenvalues  satisfy
\begin{equation*}
	\lim_{\tau\rightarrow0^+}\lambda_i^+(\tau)=\frac{\mu_i^2}{\gamma_0},
	\qquad \quad \lim_{\tau\rightarrow0^+}\lambda_i^-(\tau)=-\infty,
\end{equation*}
if $\displaystyle\lim_{\tau\to0^+}\gamma_\tau=:\gamma_0>0$. In particular, if $\gamma_0 = 1$, $\lambda_i^+(\tau)$ converge to the eigenvalues of the parabolic case, as expected.
\end{proof}

To conclude this section, we use singular perturbation theory to compare, for $\tau$ small,
the solutions of the system \eqref{h-system} and the ones of
\begin{equation}\label{h-par}
	\begin{cases}
	\bm h'=\bm\eta, \\
	\bm\eta=\mathcal{P}^*(\bm h),
	\end{cases}
\end{equation}
that is obtained by substituting $\tau=0$ in \eqref{h-system}, assuming that $\gamma_\tau\to1$ as $\tau\to0$. 
Denote by $(\bm h_p,\bm\eta_p)$ the solutions of \eqref{h-par}; $\bm h_p$ is solution of the system $\bm h'=\mathcal{P}^*(\bm h)$, 
that describes the evolution of layer positions in the parabolic case \eqref{AllenCahn}. 
Set
\begin{equation*}
	\E_\tau(t):=\gamma_\tau|\bm h(t)-\bm h_p(t)|+\tau|\bm\eta(t)-\bm\eta_p(t)|. 
\end{equation*}
A general theorem of Tihonov on singular perturbations could be applied to systems \eqref{h-system}-\eqref{h-par}. 
Specifically for the system \eqref{h-system} we have the following result.
\vspace{0.2cm}
\begin{thm}\label{thm:tau0}
Let $(\bm h,\bm \eta)$ be a solution of \eqref{h-system} and $(\bm h_p,\bm\eta_p)$ a solution of \eqref{h-par},
with $ \bm h(t),\bm h_p(t)\in\Omega_\rho$ for any $t\in[0,T]$. 
Then, there exists $C>0$ (independent of $\tau$) such that 
\begin{equation}\label{E(t)<}
	\E_\tau(t)\leq C(\E_\tau(0)+1-\gamma_\tau+\tau), \qquad \quad \mbox{ for } t\in[0,T].
\end{equation}
Moreover,
\begin{align}
	\int_0^T|\bm\eta(t)-\bm\eta_p(t)|dt &\leq\frac C{\gamma_\tau}(\E_\tau(0)+1-\gamma_\tau+\tau), \label{eta-L1}\\
	|\bm\eta(t)-\bm\eta_p(t)| &\leq\frac C{\gamma_\tau}(\E_\tau(0)+1-\gamma_\tau+\tau), \qquad \quad \mbox{ for } t\in[t_1,T],\label{eta-inf}
\end{align}
for all $t_1\in(0,T)$. 
In particular, from \eqref{E(t)<}, \eqref{eta-L1} and  \eqref{eta-inf}, it follows that, if $\gamma_\tau\to1$ and $\E_\tau(0)\rightarrow0$ as $\tau\rightarrow0$, then
\begin{equation*}
	\lim_{\tau\rightarrow0}\sup_{t\in[0,T]}|\bm h(t)-\bm h_p(t)|
	=\lim_{\tau\rightarrow0} \int_0^T|\bm\eta(t)-\bm\eta_p(t)|dt
	=\lim_{\tau\rightarrow0}\sup_{t\in[t_1,T]}|\bm\eta(t)-\bm\eta_p(t)|=0,
 \end{equation*}
for any $t_1\in(0,T)$.
\end{thm}

\begin{proof}
For $t\in[0,T]$, define
\begin{equation*}
	\bm\delta_{\bm h}(t):=\bm h(t)-\bm h_p(t), \qquad \quad
	\bm\delta_{\bm\eta}(t):=\bm\eta(t)-\bm\eta_p(t).
\end{equation*}
By hypotheses, $\bm h(t),\bm h_p(t)\in\Omega_\rho$ for $t\in[0,T]$, so $|\bm h(t)|$ and $|\bm h_p(t)|$ are uniformly bounded in $[0,T]$. 
Since $\mathcal{P}^*(\bm h)$ is a regular function of $\bm h$, there exists $C>0$ such that 
\begin{equation}\label{P*-stime}
	|\mathcal{P}^*(\bm h_p)|\leq C, \qquad |J\mathcal{P}^*(\bm h_p)|\leq C, \qquad |\mathcal{P}^*(\bm h_p+\bm\delta_{\bm h})-\mathcal{P}^*(\bm h_p)|\leq C|\bm\delta_{\bm h}|,
\end{equation}
for all $t\in[0,T]$. 
Here and in what follows, $C$ is a positive constant independent of $\tau$ whose value may change from line to line. 
We have
\begin{equation*}
	\bm\delta_{\bm h}'=\bm\eta-\bm\eta_p, \qquad \quad 
	\tau\bm\delta_{\bm\eta}'=\mathcal{P}^*(\bm h_p+\bm\delta_p)-\gamma_\tau\mathcal{P}^*(\bm h_p)-\gamma_\tau\bm\delta_{\bm\eta}-\tau J\mathcal{P}^*(\bm h_p)\mathcal{P}^*(\bm h_p).
\end{equation*}
Since $\displaystyle\frac d{dt}|\bm\delta|=\frac{\bm\delta'\cdot\bm\delta}{|\bm\delta|}$ for any $\bm\delta(t)\in\mathbb{R}^N$, 
using \eqref{P*-stime} and Cauchy--Schwarz inequality, we obtain
\begin{equation*}
	\frac d{dt}|\bm\delta_{\bm h}|\leq|\bm\delta_{\bm\eta}|, \qquad 
	\tau\frac d{dt}|\bm\delta_{\bm\eta}|\leq C\gamma_\tau|\bm\delta_{\bm h}|+(1-\gamma_\tau)C-\gamma_\tau|\bm\delta_{\bm\eta}|+\tau C.
\end{equation*}
Summing, one has
\begin{equation*}
	\frac d{dt}\left(\gamma_\tau|\bm\delta_{\bm h}|+\tau|\bm\delta_{\bm\eta}|\right)\leq C\gamma_\tau|\bm\delta_{\bm h}|+C(1-\gamma_\tau+\tau),
\end{equation*}
and so
\begin{equation}\label{d/dt E<}
	\frac d{dt}\E_\tau(t)\leq C\left(\E_\tau(t)+1-\gamma_\tau+\tau\right),  \qquad \quad \mbox{ for } t\in[0,T].
\end{equation}
Integrating \eqref{d/dt E<} and applying Gr\"onwall's Lemma, we obtain \eqref{E(t)<}. 
In particular, from \eqref{E(t)<}, it follows that 
\begin{equation}\label{delta_h}
	\gamma_\tau|\bm\delta_{\bm h}(t)|\leq C(\E_\tau(0)+1-\gamma_\tau+\tau),  \qquad \quad\qquad \mbox{ for } t\in[0,T].
\end{equation}
Substituting \eqref{delta_h} in the equation for $\bm\delta_{\bm\eta}$, we obtain
\begin{equation}\label{delta_eta}
	\tau\frac{d}{dt}|\bm\delta_{\bm\eta}|\leq-\gamma_\tau|\bm\delta_{\bm\eta}|+C(\E_\tau(0)+1-\gamma_\tau+\tau).
\end{equation}
Integrating \eqref{delta_eta}, we obtain \eqref{eta-L1}. 
Furthermore, for \eqref{delta_eta}, we have 
\begin{equation*}
	\frac{d}{dt}\left(\tau e^{\gamma_\tau t/\tau}|\bm\delta_{\bm\eta}(t)|\right)\leq C(\E_\tau(0)+1-\gamma_\tau+\tau)e^{\gamma_\tau t/\tau},
\end{equation*}
and so
\begin{align*}
	|\bm\delta_{\bm\eta}(t)|&\leq\frac C{\gamma_\tau}(\E_\tau(0)+1-\gamma_\tau+\tau)\bigl(1-e^{-\gamma_\tau t/\tau}\bigr)+|\bm\delta_{\bm\eta}(0)|e^{-\gamma_\tau t/\tau} \\
	& \leq\frac C{\gamma_\tau}(\E_\tau(0)+1-\gamma_\tau+\tau) + \E_\tau(0)\frac{e^{-\gamma_\tau t/\tau}}{\tau},
\end{align*}
for $t\in[0,T]$.
Therefore, for any fixed $t_1\in(0,T)$, we obtain \eqref{eta-inf}.
\end{proof}

\section*{Acknowledgements}
We thank the anonymous Referees for the careful review and for the suggestions which help us to improve our paper.

\end{document}